\documentclass[hidelinks,onefignum,onetabnum]{siamart250211}

\usepackage[utf8]{inputenc}

\usepackage[english]{babel}

\usepackage{amssymb}
\usepackage{braket,amsfonts}
\usepackage{mathrsfs}
\usepackage{mathtools}
\usepackage{enumerate}
\usepackage{enumitem}
\usepackage{bbm}
\usepackage{hyperref}
\usepackage{extarrows}
\usepackage{comment}
\usepackage{soul}
\usepackage{xcolor}
\usepackage{graphicx}
\usepackage[export]{adjustbox} 
\usepackage{cleveref} 
\usepackage{cite}

 \newtheorem{assump}{Assumption}
\newenvironment{myassump}[2][]
  {%
   \renewcommand\theassump{#2}%
   \if\relax\detokenize{#1}\relax
     \begin{assump}%
   \else
     \begin{assump}[#1]%
   \fi
  }
  {\end{assump}}

\newtheorem{claim}{Claim}
\newtheorem{remark}{Remark}

\usepackage{scalerel}

\theoremstyle{empty}

\DeclareMathOperator*{\argmin}{arg\,min}

\newcommand{\R}{\mathbb{R}}
\def\P{{\mathcal P}}
\newcommand{\bR}{\mathbf{R}}
\newcommand{\bC}{\mathbf{C}}
\newcommand{\bP}{\mathbf{P}}
\newcommand{\bB}{\mathbf{B}}
\newcommand{\bb}{\mathbf{b}}
\newcommand{\bc}{\mathbf{c}}
\newcommand{\bp}{\mathbf{p}}
\newcommand{\bd}{\mathbf{d}}
\newcommand{\bA}{\mathbf{A}}
\newcommand{\bK}{\mathbf{K}}
\def\bf{\mathbf{f}}
\def\bg{\mathbf{g}}
\newcommand{\bu}{\mathbf{u}}

\newcommand{\bv}{\mathbf{v}}

\newcommand{\bw}{\mathbf{w}}
\newcommand{\bx}{\mathbf{x}}
\newcommand{\by}{\mathbf{y}}
\newcommand{\bz}{\mathbf{z}}
\newcommand{\bnu}{\boldsymbol{\nu}}
\newcommand{\bsigma}{\boldsymbol{\sigma}}
\newcommand{\bGamma}{\boldsymbol{\Gamma}}
\newcommand{\bgamma}{\boldsymbol{\gamma}}
\newcommand{\la}{\langle}
\newcommand{\ra}{\rangle}

\usepackage{amsmath, amssymb}
\usepackage{bbm} 
\usepackage{tikz}
\usetikzlibrary{positioning, arrows.meta, calc, shapes.misc}

\author{
  Katy Craig\thanks{Department of Mathematics, University of California,
    Santa Barbara, (\email{kcraig@ucsb.edu}).}
  \and
  Benjamin Faktor\thanks{Department of Mathematics, University of California,
    Los Angeles, (\email{benjaminfaktor@math.ucla.edu}).}
  \and
  Benjamin Nachman\thanks{Department of Particle Physics and Astrophysics, Stanford University, and Fundamental Physics Directorate, SLAC National Accelerator Laboratory (\email{nachman@stanford.edu}).}
}

\begin{document}

\title{Unfolding with a Wasserstein Loss } 
\maketitle

\tableofcontents

\begin{abstract}
Data unfolding --- the removal of noise or artifacts from
measurements --- is a fundamental task across the experimental sciences. Of particular interest in the present work are applications of data unfolding in  physics, in which context the dominant approach is Richardson-Lucy (RL) deconvolution. The classical RL approach  aims to find   denoised data that, once passed through the noise model,  is as close as possible to the measured data, in terms of   Kullback--Leibler (KL) divergence. Fundamental to this approach is the hypothesis that the support of the measured data overlaps with the output of the noise model, so that the KL divergence correctly captures their similarity. In practice, this hypothesis is typically enforced by binning the measured data and noise model, introducing numerical error into the unfolding process.

As a counterpoint to classical binned methods for unfolding, the present work studies an alternative formulation of the unfolding problem, using a Wasserstein loss instead of the KL divergence to quantify the similarity between the measured data and the output of the noise model.   We establish sharp conditions for 
existence and uniqueness of optimizers; as a consequence we answer open questions of Li, et al.  \cite{Li25},  regarding   necessary conditions for  uniqueness in the case of transport map noise models. Following these theoretical results, we then develop a provably convergent
generalized Sinkhorn algorithm to compute approximate optimizers. Our algorithm requires only empirical observations of the noise model and measured data and scales
with the size of the data, rather than the ambient dimension.  Numerical
experiments on one- and two-dimensional problems inspired by jet mass
unfolding in particle physics demonstrate that the optimal transport
approach offers robust, accurate performance compared to classical
Richardson--Lucy deconvolution, particularly when binning artifacts are
significant.\end{abstract}



\section{Introduction} 

\emph{Data unfolding} is a fundamental problem across the   sciences, in which one seeks to remove noise or   artifacts from data. The ``true''   denoised data  can then be   compared with the output of other experiments or numerical simulations. 
Given a measurable space $\mathcal{X}$ and $\mathcal{Y} = \mathbb{R}^d$, an  unfolding problem is specified in terms of three components: a noise operator $N: \mathcal{P}(\mathcal{X}) \to \mathcal{P}(\mathcal{Y})   $ that sends true data $\sigma \in \mathcal{P}(\mathcal{X})$ to   corresponding noisy measurements $N(\sigma)   \in \mathcal{P}(\mathcal{Y})$, the measured data $ {\nu} \in \mathcal{P}(\mathcal{Y})$, and a  loss function $D:   \mathcal{P}(\mathcal{Y}) \times \mathcal{P}(\mathcal{Y}) \to \R \cup \{+\infty\}$.  
One aims to solve
\begin{equation}
\label{general denoising problem}
\sigma_* \in \argmin_{\sigma } D(N(\sigma), \nu)
\end{equation}
The measure $\sigma_* \in \mathcal{P}(\mathcal{X})$ represents true data that best agrees with   observations ${\nu}$.

Of particular interest in the present work are applications to particle, nuclear, and astro physics, such as at the Large Hadron Collider, where unfolding is a critical step in measuring differential cross sections --- reaction rates that connect the data to quantum field theory predictions. In these applications, the primary role of data unfolding is to correct observations from distortions due to detector effects. The most common choice of loss function in this context is the KL divergence,
 \begin{align} \label{KLdef}
D_{\rm KL}(N(\sigma),  {\nu}) = {\rm KL}( {\nu} | N(\sigma)) 
\end{align}
and $\sigma_*$ is known as the \emph{non-parametric maximum likelihood estimator} (MLE). 
Provided that there exists  a reference measure $m \in \mathcal{P}(\mathcal{Y})$ so that  
\begin{align} \label{KLnecessities}    {\nu}   \ll N(\sigma) \ll m  \text{ and } {\rm KL}({\nu}|m)<+\infty,
\end{align} 
minimizing  $D_{\rm KL}$ is equivalent to minimizing  
$D_{\rm MLE}(N(\sigma),  {\nu}) = - \int \log \left( \frac{d N(\sigma)}{ d m} \right) d  {\nu}   .
$
Furthermore, when the measured data is empirical, ${\nu} = \frac{1}{K} \sum_{k=1}^K \delta_{\bar{y}_k}$,   $D_{\rm MLE}$ is the negative log-likelihood 
\begin{align} \label{MLEN}
D_{\rm MLE} \left(N(\sigma), \frac{1}{K} \sum_{k=1}^K \delta_{\bar{y}_k} \right) = -\frac{1}{K} \sum_{k=1}^K \log \left( \left(\frac{d N(\sigma)}{ d m} \right)(\bar{y}_k) \right).
\end{align} The standard approach to minimizing $D_{\rm MLE}$  is via Expectation-Maximization (EM), also known as Iterative Bayesian Unfolding (IBU) or Richardson-Lucy (RL) deconvolution \cite{dempster1977maximum, richardson1972bayesian, lucy1974iterative}.

To solve the unfolding problem (\ref{general denoising problem}), a   method must not only be well-suited to  the    loss function but also to   the noise operator. A classical noise operator in the case $\mathcal{X}=\mathcal{Y}$ is the mixture model with respect to  $m \in \P(\mathcal{Y})$,
\begin{align} \label{GMM}
N(\sigma) := \varphi *\sigma d m , \quad \varphi*\sigma(\cdot) : = \int \varphi(\cdot-x) d \sigma(x) , \quad \text{ for } \varphi \in C_b(\mathcal{Y})  \text{ fixed.}
\end{align}
In contrast, in particle physics, the noise operator is often only known   through simulation and is  modeled by a Markov kernel $\rho: \mathcal{X} \to \mathcal{P}(\mathcal{Y})$, where ``true'' signal at location $x$, once passed through the noise model, leads to measurements distributed as $\rho_x \in \P(\mathcal{Y})$. This more general class of noise models will be the focus of the present work. In particular,  we take   $N(\sigma) := \nu_\sigma$, where $\nu_\sigma$ satisfies
\begin{align} \label{Markovnoise}
\int f(y) d \nu_\sigma(y) = \int_\mathcal{X} \left( \int_\mathcal{Y} f(y) d \rho_x(y) \right) d \sigma(x) , \quad \forall f \in C_b(\mathcal{Y}).
\end{align}

In spite of the practical success of the KL divergence as a loss function for the unfolding problem, in order for the equivalence between  minimizing (\ref{KLdef}) and  (\ref{MLEN}) to hold, which is necessary if one seeks to compute the optimizer via the classical EM approach, one must find a   reference measure $m$ satisfying  the hypotheses (\ref{KLnecessities}).   In most particle physics applications, $m$ is chosen to be an empirical measure supported on a grid, $m = \frac{1}{M} \sum_{i=1}^M \delta_{y_i}$, and both the measured data $ {\nu}$ and   noise model $N$ are binned to enforce (\ref{KLnecessities}). However, this approach has   major limitations.  Experiments must coordinate on the bin centers ahead of time and derivative measurements, such as distribution moments, carry significant biases from binning.
%
%
Furthermore, since the computational complexity of grid-based methods scales exponentially with the dimension, to make the problem computationally tractable in practice,   $\nu$ and $N(\sigma)$ are typically replaced with lower dimensional projections,  throwing away additional information.

A recent approach to overcoming these limitations in particle physics, known as OmniFold,  uses machine-learning models to approximate the density of ${\nu}$ with respect to $N(\sigma)$, based on empirical observations, and takes these as input to the classical EM algorithm \cite{andreassen2020omnifold, Andreassen:2021zzk}. Subsequent work has incorporated the (sliced) Wasserstein distance into this framework for the reweighting step \cite{pan2024swdfold}, considered non-iterative (and thus non-EM) variants \cite{Ore:2026qgp}, along with other modifications \cite{Arratia:2021otl,Huetsch:2024quz}. However, all such methods are fundamentally based on the KL loss (\ref{KLdef}).  While the original OmniFold has been used for a growing number of measurements across experiments~\cite{Canelli:2025ybb}, fundamental challenges remain.  For example, without significant overlapping support between the observed data $\nu$ and forward-folded estimates $N(\sigma)$,   the KL loss is uninformative.  Furthermore, 
the black-box nature of the machine-learning approximation makes  theoretical guarantees   challenging. 

The goal of the present work is to overcome these limitations  by considering an alternative loss function, given by the $p$-Wasserstein metric,
\begin{align} \label{Wploss}
D_{W_p}(N(\sigma),  {\nu}) = W_p^p(N(\sigma),  {\nu}) , \quad  p \geq 1.
\end{align} 
This loss function removes the requirement that $\nu$ be absolutely continuous with respect to $N(\sigma)$, as in the KL case. Furthermore, it quantifies the difference between $N(\sigma)$ and $\nu$ in a way that captures the geometry of the distribution of mass in the underlying space $\mathcal{Y}$, rather than just the differences in relative  magnitude where $N(\sigma)$ and $\nu$ overlap, as with KL. This distinction is especially important in particle physics applications, where the underlying space $\mathcal{Y}$ often has a strong geometric interpretation---for example, the location at which particles hit the detector.

The main contributions of the present work are to establish existence and uniqueness of solutions to the   unfolding problem \eqref{general denoising problem} with $p$-Wasserstein loss and develop a numerical method to   approximate minimizers.  For noise models given by Markov kernels, we give sharp conditions under which solutions  exist (Theorem \ref{thm:existence}) and are unique (Theorem \ref{thm:uniqueness}). The conditions themselves are natural in the context of optimal transport, and  the main novelty  is the demonstration of sharpness. Next,   we develop a provably convergent generalized Sinkhorn algorithm for   approximating minimizers, which merely requires empirical observations of the Markov kernel, rather than an explicit functional form.

\subsection{Relation to Previous Work}
Several previous works have considered the role of $p$-Wasserstein distances in unfolding problems. The first line of work in this direction arises in the study of statistical estimators. Bassetti, Bodini, and Regazzini \cite{bassetti2006minimum} considered the Wasserstein loss (\ref{Wploss}) in the case when the measured data $\nu$ is an empirical observation of an underlying continuum measure of the form $\tilde{\nu} = N(\sigma_\text{true})$. The authors developed sufficient conditions for existence of an optimizer, which they called the \emph{minimal Kantorovich estimator}, as well as its consistency as the empirical observations of the measured data converge to $\tilde{\nu}$. Subsequent work in this line of inquiry considered the case when the state space $\mathcal{X}$ is finite. Bernton et al. \cite{bernton2019parameter}   succeeded in  removing the hypothesis $\tilde{\nu} = N(\sigma_\text{true})$, again developing sufficient conditions for existence and consistency. Bernton, et al. also introduced a numerical approach for computing the minimizer based on Monte-Carlo expectation maximization.  Finally, the  Wasserstein loss likewise arises in Garc\'ia-Trillos, Jaffe, and Sen's work on the Cram\'er-Rao theory of unbiased estimation, in which case they refer to the solution of the unfolding problem as  the \emph{Wasserstein projection estimator} (WPE) \cite{trillos2025wasserstein}. Under sufficient regularity hypotheses on the noise model, the authors show that the WPE is \emph{asymptotically sensitivity-efficient}. 

A second line of research that considered $p$-Wasserstein distances in unfolding problems arises in the context of learning generative models. Assuming the noise operator is of the form $N(\sigma) = g_\sigma \# \sigma$, where $g_\sigma$ is a differentiable function parametrized by $\sigma$, such as a neural network, Genevay, Peyr\'e, and Cuturi developed a numerical scheme for approximating optimizers via auto-differentiation  \cite{genevay2018learning}. In subsequent work, the same authors observed that an alternative numerical scheme, based on a dual formulation of the Wasserstein metric and  the ansatz that the Kantorovich potential belongs to a parametric class, leads to the Wasserstein GAN problem \cite{genevay2017gan}. Recent work by Akyildiz, Girolami, Vadeboncouer, and Stuart \cite{vadeboncoeur2025efficient, akyildiz2025efficient} consider the case of noise models given by the pushforward via a transport map and convolution with a mollifier, $N(\sigma) = \varphi* (t \# \sigma)$. When the measured data is concentrated at a single point $\nu = \delta_y$ for $y \in \mathcal{Y}$, minimizers of the unfolding problem are characterized, and for general $\nu$, a numerical method is developed using autodifferentiation to perform gradient descent on the sliced Wasserstein metric. 

The present work differs from these previous works in two respects. First, we study  existence as well as uniqueness of solutions for general measured data $\nu \in \mathcal{P}(\mathbb{R}^d)$ and an arbitrary Polish space $\mathcal{X}$, so that so that (\ref{general denoising problem}) is an optimization problem over an infinite dimensional space of measures. Furthermore, beyond developing sufficient conditions for solutions to exist and be unique, we show that our conditions are sharp, in the sense that the result fails in general if any condition is removed. A second key difference of the present work is that  our numerical approach, based on a generalized Sinkhorn algorithm, computes an approximate minimizer via alternating Bregman projections. A major benefit of this perspective is that it is well-suited to the empirically observed measures and noise models arising in particle physics applications, without requiring additional hypotheses that the measures, noise models, or Kantorovich potentials belong to a prescribed parametric class. Furthermore, existing results in the optimization literature guarantee convergence to a unique minimizer and provide a rate of convergence in dual variables; see Remark \ref{rate of convergence}.

Two recent works that are closely aligned with the present work are recent work by   Li, et al. \cite{li2024stochastic, Li25} and Lasserre \cite{La24}. Li, et al. investigate the   unfolding problem \eqref{general denoising problem} with  noise operator $N$ defined via pushforward, $N(\sigma) = t \# \sigma$, which we call \emph{transport map noise}. Several  choices of loss function are considered, including the $p$-Wasserstein loss. The authors develop sufficient conditions   for existence of minimizers and, under certain hypotheses on $t$, an explicit characterization of one of the minimizers. The question of uniqueness is left open. Continuing in this line of research, our Corollaries \ref{cor:existence}-\ref{cor:uniqueness} provide sharp conditions that ensure existence and uniqueness of minimizers in the case of transport map noise.

  Lasserre \cite{La24} considers the unfolding problem  \eqref{general denoising problem} with 2-Wasserstein metric and a Gaussian noise model $N(\sigma)$, where $\mathcal{X}$ is a compact set of mean and variance parameters. Lasserre characterizes the unique minimizer in terms of a moment relaxation of the unfolding problem. A special case of our main existence result, Theorem \ref{thm:existence}, generalize Lasserre's result on existence of minimizers to when $\mathcal{X}$ is an arbitrary nonempty Polish space; see Remark \ref{LasserreRemark}.

 
Finally, while the present work focuses on unfolding with a  Wasserstein loss, there are close connections with  unfolding with the KL divergence (\ref{KLdef}). Notably, Rigollet and Weed \cite{rigollet2018entropic} proved that, when the noise is given by a Gaussian mixture model---equation (\ref{GMM}) with  $\phi$ a Gaussian with standard deviation $\epsilon>0$---the \emph{maximum likelihood estimator} $\sigma$ that solves the unfolding problem with loss given by (\ref{MLEN}) is also the measure that minimizes $\sigma \mapsto W_{2,\epsilon}^2(\sigma, \nu)$, where $W_{2,\epsilon}$ is the \emph{entropically regularized} 2-Wasserstein distance; see equation \eqref{entropicWp}. However, this result is largely complementary to the present study, since our focus is on noise models given by general Markov kernels $N(\sigma) = \nu_\sigma$.

\subsection{Main Results}
Let $\mathcal{X}$ be a nonempty Polish space and $\mathcal{Y} = \mathbb{R}^{d}$. Let $\mathcal{P}(\mathcal{X})$ denote the set of Borel probability measures on $\mathcal{X}$ and $\mathcal{P}_p(\mathcal{Y})$ denote the elements of $\mathcal{P}(\mathcal{Y})$ with finite $p$th moment $M_p$.
Let $\rho:\mathcal{X} \rightarrow \mathcal{P}(\mathcal{Y})$ be a Markov kernel, i.e. such that the evaluation of every fixed Borel set is a Borel-measurable function. 
For each $\sigma \in \mathcal{P}(\mathcal{X})$, we consider the noise model $N(\sigma):= \nu_\sigma$, where $\nu_\sigma$ is given by \eqref{Markovnoise}.

Let $p \geq 1$ and $\nu \in \mathcal{P}_p(\mathcal{Y})$. We consider the denoising problem
\begin{equation}
\label{minimization problem}
\sigma_* := \arg \min_{\sigma \in \mathcal{P}(\mathcal{X})} W_p^p(\nu_\sigma, \nu).
\end{equation}
The goal of the present work is to identify sharp hypotheses on the noise model $\rho$ and the measured data $\nu$ under which solutions exist and are unique. For the noise model, we consider the following  continuity and coercivity assumptions on the Markov kernel.
\begin{myassump}{(CTY)} \label{CTYas}
$x \mapsto \rho_{x}$ is continuous in the narrow topology.
\end{myassump}

\begin{myassump}{(CRC)} \label{CRCas}
$x \mapsto M_p(\rho_{x})$ has totally bounded sublevel sets.
\end{myassump}
These assumptions are satisfied for most reasonable choices of noise models that arise in applications. For example, if $\mathcal{X}$ is finite, as is the case for collider data and the discrete setting explored in \S\ref{section:numerical method}, then trivially conditions \ref{CTYas} and \ref{CRCas} hold.  

We also identify an injectivity condition on $\rho$ that characterizes uniqueness of minimizers:

\begin{myassump}{(INJ)} \label{INJas}
$\sigma \mapsto \nu_\sigma$ is injective on 
\begin{align} \label{Adef}
\mathscr{A} := \{\sigma \in \mathcal{P}(\mathcal{X}): \sigma \text{ minimizes } \eqref{minimization problem}\}.
\end{align}
\end{myassump}

We now state our main well-posedness results. Let $\mathcal{L}^{d}$ denote the  Lebesgue measure on $\mathbb{R}^d$. We consider $\sigma_*$ to be a \textit{minimizer} of (\ref{minimization problem}) if $W_p^p(\nu_{\sigma_*},\nu) = \inf_\sigma  W_p^p(\nu_{\sigma},\nu)$, and we call (\ref{minimization problem}) \textit{feasible} if $\inf_{\sigma} W_p^p(\nu_{\sigma},\nu) <+\infty$. 

\begin{theorem}
\label{thm:existence}
If   \ref{CTYas} and \ref{CRCas} hold,
 a minimizer of \eqref{minimization problem} exists.
\end{theorem}
\begin{theorem}
\label{thm:uniqueness}
Suppose   $p>1$
and $\nu \ll \mathcal{L}^{d}$. Then minimizers of \eqref{minimization problem}  are unique if and only if \ref{INJas} holds and \eqref{minimization problem} is feasible.
\end{theorem}

In Propositions \ref{necessity of existence hypotheses examples} and \ref{examples showing necessity of uniqueness hypotheses} we show by counterexample that none of the hypotheses of Theorems \ref{thm:existence} and \ref{thm:uniqueness} may be omitted: see   Figure \ref{fig:well posedness summary for general denoising problem} for a summary of these results.\footnote{Note that it is immediate that \ref{INJas} and
  feasibility of (\ref{minimization problem}) are necessary for  Theorem \ref{thm:uniqueness}; for brevity, these are omitted from the figure. Furthermore,  the only use of the assumption $\nu \ll \mathcal{L}^{d}$ in Theorem \ref{thm:uniqueness} is to ensure the strict convexity of $W_p^p$ along linear interpolations, which holds as long as the optimal transport plan from $\nu$ to any other measure is induced by a map. Consequently, one may relax this assumption to only require that $\nu$ does not concentrate mass on sets of Hausdorff dimension $d-1$  \cite{Ga96}.}
\begin{figure}[h]
\[
{\setlength{\arraycolsep}{4pt}
\begin{array}{|c|c|c|}
\hline
\exists \eqref{minimization problem} & \mathrm{\ref{CTYas}} & \neg\mathrm{\ref{CTYas}} \\
\hline
\mathrm{\ref{CRCas}} 
  & \begin{array}{c}
      \text{Yes} \\
      \text{Thm. \ref{thm:existence}}
    \end{array}
  & \begin{array}{c}
      \text{No} \\
      \text{Prop.\ \ref{necessity of existence hypotheses examples}}
    \end{array}
\\
\hline
\neg\mathrm{\ref{CRCas}} 
  & \begin{array}{c}
      \text{No} \\
      \text{Prop.\ \ref{necessity of existence hypotheses examples}}
    \end{array}
  & {}
\\
\hline
\end{array}
}
\ \ \ \ \
{\setlength{\arraycolsep}{4pt}
\begin{array}{|c|c|c|}
\hline
! \eqref{minimization problem} & p>1 & \neg(p>1) \\
\hline
\nu \ll \mathcal{L}^{d}
  & \begin{array}{c}
      \text{Yes} \\
      \text{Thm.\ \ref{thm:uniqueness}}
    \end{array}
  & \begin{array}{c}
      \text{No} \\
      \text{Prop.\ \ref{examples showing necessity of uniqueness hypotheses}}
    \end{array}
\\
\hline
\neg(\nu \ll \mathcal{L}^{d})
  & \begin{array}{c}
      \text{No} \\
      \text{Prop.\ \ref{examples showing necessity of uniqueness hypotheses}}
    \end{array}
  & {}
\\
\hline
\end{array}
}
\]
\caption{Summary of main results on existence and uniqueness of optimizers for unfolding with a Wasserstein loss. Left: sharp conditions for existence of minimizers of (\ref{minimization problem}) Right: sharp conditions for uniqueness of minimizers of (\ref{minimization problem}). }
\label{fig:well posedness summary for general denoising problem}
\end{figure}

\begin{remark} \label{LasserreRemark}
Lasserre \cite{La24} studies the problem \eqref{minimization problem} where the noise model $\rho$ assigns points $(\alpha, \Sigma)$ of a compact parameter set $\mathcal{X} \subseteq \mathbb{R}^d \times \mathbb{R}^{d(d+1)/2}$ to the measure $\rho_{(\alpha, \Sigma)}$ on $\mathcal{Y} = \mathbb{R}^d$ with density w.r.t. $\mathcal{L}^d$ given by a multivariate Gaussian of mean $\alpha$ and covariance $\Sigma$. In this way, $\nu_{\sigma}$ is a mixture of Gaussians according to the distribution $\sigma$ of mean and variance parameters. Seeing as this noise model satisfies the conditions of Theorem \ref{thm:existence}, as an immediate corollary we generalize Lasserre's main existence result to nonempty Polish subspaces of $\mathbb{R}^d \times \mathbb{R}^{d(d+1)/2}$. 
\end{remark}

In Propositions \ref{(CTY1), (CTY2), and (CRC) sufficient conditions}
and \ref{(INJ) sufficient conditions}, we provide examples of noise models $\rho$ that satisfy the hypotheses of our main  Theorems \ref{thm:existence} and \ref{thm:uniqueness}.
As a particular case of these results, we resolve an open problem asked in Li, et al. \cite{Li25} on the well-posedness of the denoising problem 
\begin{equation}
\label{Li et al denoising problem}
\sigma_* := \arg \min_{\sigma \in \mathcal{P}(\mathcal{X})} W_p^p(\phi \# \sigma, \nu)
\end{equation}
where $\phi:\mathcal{X} \rightarrow \mathcal{Y}$ is measurable. Indeed, \eqref{Li et al denoising problem} is of the form \eqref{minimization problem} with noise model $\rho_{x} = \delta_{\phi(x)}$, which we call \textit{transport map} noise, where   $\delta_\alpha$ denotes the Dirac delta measure at $\alpha \in \mathbb{R}^{d}$.

\begin{corollary}
\label{cor:existence}
Suppose that $\phi$ is continuous and $|\phi|$ has totally bounded sublevel sets.
Then a minimizer of \eqref{Li et al denoising problem} exists.
\end{corollary}

\begin{corollary}
\label{cor:uniqueness}
Suppose $p>1$, $\nu \ll \mathcal{L}^{d}$, and $\phi$ is injective. Then minimizers of \eqref{Li et al denoising problem}, if they exist, are unique if and only if \eqref{Li et al denoising problem} is feasible. 
\end{corollary}

The above hypotheses on $\phi$ imply \ref{CTYas}, \ref{CRCas}, and \ref{INJas}; see Propositions \ref{(CTY1), (CTY2), and (CRC) sufficient conditions} and \ref{(INJ) sufficient conditions}. In Figure \ref{fig:well posedness summary for Li et al denoising problem} we summarize our well-posedness results for \eqref{Li et al denoising problem}, again omitting the trivial uniqueness conditions of the injectivity of $\phi$ and feasibility of \eqref{Li et al denoising problem}. We abbreviate ``$\phi$ is continuous'' as ``$\phi$ cont.'' and ``$|\phi|$ has totally bounded sublevel sets'' as ``$\phi$ crc.''.

\begin{figure}[h]
\[
{\setlength{\arraycolsep}{4pt}
\begin{array}{|c|c|c|}
\hline
\exists \eqref{Li et al denoising problem} & \text{$\phi$ } \mathrm{cont.} & \neg(\text{$\phi$ } \mathrm{cont.}) \\
\hline
\mathrm{\text{$\phi$ } \mathrm{crc.}} 
  & \begin{array}{c}
      \text{Yes} \\
      \text{Cor. \ref{cor:existence}}
    \end{array}
  & \begin{array}{c}
      \text{No} \\
      \text{Prop.\ \ref{necessity of existence hypotheses examples}}
    \end{array}
\\
\hline
\neg(\text{$\phi$ } \mathrm{crc.})
  & \begin{array}{c}
      \text{No} \\
      \text{Prop.\ \ref{necessity of existence hypotheses examples}}
    \end{array}
  & {}
\\
\hline
\end{array}
}
\ \ \ \ \ 
{\setlength{\arraycolsep}{4pt}
\begin{array}{|c|c|c|}
\hline
! \eqref{Li et al denoising problem} & p>1 & \neg(p>1) \\
\hline
\nu \ll \mathcal{L}^{d}
  & \begin{array}{c}
      \text{Yes} \\
      \text{Cor.\ \ref{cor:uniqueness}}
    \end{array}
  & \begin{array}{c}
      \text{Yes ($d=1$)} \\
      \text{Prop.\ \ref{d=1}}
    \end{array}
\\
\hline
\neg(\nu \ll \mathcal{L}^{d})
  & \begin{array}{c}
      \text{No} \\
      \text{Prop.\ \ref{examples showing necessity of uniqueness hypotheses for transport map noise}}
    \end{array}
  & {}
\\
\hline
\end{array}
}
\]
\caption{Summary of existence and uniqueness results for unfolding with a Wasserstein loss, in the special case of transport map noise. Left: sharp conditions for existence of minimizers of (\ref{Li et al denoising problem}). Right:   conditions for uniqueness of minimizers of (\ref{Li et al denoising problem}). }
\label{fig:well posedness summary for Li et al denoising problem}
\end{figure}

Note that, in the case of general noise models $\rho_x$ treated in Theorem \ref{thm:uniqueness}, uniqueness fails when $p=1$. However, in the present case of transport map noise $\rho_x = \delta_{\phi(x)}$ uniqueness holds additionally in the case $p = 1$, $d=1$; see Proposition \ref{d=1}. We leave the study of uniqueness for $p=1$, $d>1$ to future work. 

We conclude by developing a numerical method for computing approximate solutions to \eqref{minimization problem}. Due to the inherent instability of \eqref{minimization problem}, in the sense that, for many noise models, there exist $\sigma$ and $\sigma'$ ``far apart'' in $W_p$ for which $\nu_{\sigma} \approx \nu_{\sigma'}$ in $W_p$, entropic regularization is used to select an approximate minimizer. At the continuum level, this leads to the following optimization problem, for a fixed regularization parameter $0 < \varepsilon \ll 1$ and reference measures $\hat{m}, \hat{\sigma}$:
\begin{align}
\label{entropic minimization problem}
\sigma_\varepsilon^* &:= \argmin_{\sigma \in \mathcal{P}(\mathcal{X})} W_{p,\varepsilon}^p(\nu_\sigma, \nu) + \varepsilon {\rm KL}(\nu_\sigma | \hat{m}) + \varepsilon {\rm KL}(\sigma | \hat{\sigma}) 
\end{align}
where $W_{p, \varepsilon}:\mathcal{P}(\mathcal{Y}) \times \mathcal{P}(\mathcal{Y}) \rightarrow [0,\infty]$ denotes the entropically regularized $p$-Wasserstein distance, given by
\begin{align} \label{entropicWp}
W_{p,\varepsilon}^p(\mu_1, \mu_2) := \inf_{\Gamma \in \Pi(\mu_1, \mu_2)} \int_{\mathcal{Y} \times \mathcal{Y}} |y' - y|^p d\Gamma(y', y) + \varepsilon {\rm KL}(\Gamma | \mu_1 \otimes \mu_2).
\end{align}
This choice of regularization is motivated by the fact that minimizers can be approximated by a provably convergent numerical scheme. Note that (\ref{entropic minimization problem}) is one of many equivalent entropic regularizations that lead to the same numerical approach; see Remark \ref{choiceOfReferenceMeasure}.

In our numerical method, we focus on the case when both the measured data $\nu$ and the values of the noise model $\rho_{x}$ are finitely supported measures, as is the case in the motivating applications in particle physics. To solve (\ref{entropic minimization problem}) numerically, we suppose in addition that $\hat{m}, \hat{\sigma}$ are empirical measures. The locations on which $\hat{\sigma}$ is supported encode a prior on the distribution of an optimal $\sigma_\varepsilon^*$. 
In this setting, (\ref{entropic minimization problem}) reduces to a fully discrete convex optimization problem with linear constraints.

Leveraging this perspective, in section \ref{section:numerical method}, we describe an iterative method based on Bregman projections to approximate the minimizer, with existing results in the optimization literature providing quantitative rates of convergence to optimum; see Remark \ref{rate of convergence}. Furthermore, there are equivalent, but  lower dimensional, generalized Sinkhorn-type iterations that only require iterating two vectors whose lengths are the dimensions of this matrix. The scheme scales only with the number of points in the support of  the measurement $\nu$ and the noisy data $\nu_\sigma$, with no intrinsic dependence on the dimension of the underlying space. 

Finally, in section \ref{section numerical results}, we compare the performance of our optimal transport (OT) unfolding method to the classical Richardson-Lucy (RL) method. In the RL case, to enforce the required absolute continuity hypothesis (\ref{KLnecessities}), we discretize the measured data $\nu$ and the Markov kernel $\rho_x$ on a grid. We consider both one dimensional problems, based on discretizations of simple continuum models, and two dimensional problems, based on a physically motivated jet mass unfolding problem. Our numerical results demonstrate that OT unfolding has robust performance with respect to the Sinkhorn regularization parameter $\epsilon$, in contrast to RL, which depends strongly and non-monotonically on the number of bins in the grid discretization. While RL requires fewer iterations to converge, OT generally attains the best accuracy, where accuracy is measured in terms of the size of $W_2(\nu_\sigma,\nu)$. The improved accuracy of OT is greatest when the discretization of the Markov kernel or measured data is coarse, so that there is little overlap between the measures. In this case, the absolute continuity hypothesis (\ref{KLnecessities}) fails badly, and coarse binning is required to restore it, introducing discretization errors into the RL unfolding method. Finally, in the case of the two dimensional jet mass unfolding problem, we also compare various  observables of the exact solution $\sigma_*$ to the OT and RL approximations, showing that the improved accuracy we observed in terms of $W_2(\nu_\sigma,\nu)$ persists for   physically relevant observables.
  
\subsection{Future work}
There are several directions for future work. In terms of analysis of the unfolding problem (\ref{minimization problem}), it would be interesting to investigate the stability of $\sigma_*$ and $\nu_{\sigma_*}$ subject to perturbations of $\nu$ and $\rho_x$. In terms of the numerical method, it is an open question to determine conditions on the convergence of discrete to continuum measures so that, as the Sinkhorn regularization is removed, minimizers of the discrete unfolding problem converge to minimizers of the original continuum unfolding problem (\ref{minimization problem}). In terms of applications in particle physics, in future work, we aim to apply to the method to a range of recent unbinned measurements; see \cite{canelli2026practical} and the references therein. Hadronic final states, where multidimensional and moment measurements are critical, particularly stand to benefit from our approach.

\subsection{Outline}
The paper is organized as follows. In \S \ref{sec:well-posedness}, we prove the existence and uniqueness of solutions to the OT unfolding problem \eqref{minimization problem} under hypotheses that we demonstrate to be sharp. \S \ref{section:numerical method} describes our numerical method for solving this problem  in the discrete setting. In \S \ref{section numerical results}, we give several numerical examples, comparing the performance of our numerical method with the classical Richardson-Lucy approach.

\section{Well-posedness of OT Unfolding}
\label{sec:well-posedness}
\ 
\subsection{Existence of minimizers}
\label{subsection-existence} 
We now prove Theorem \ref{thm:existence} on the existence of minimizers of \eqref{minimization problem} via the direct method of the calculus of variations, 
first showing lower semicontinuity of the objective functional and then showing compactness of a minimizing sequence. 
In the examples that follow we demonstrate that the hypotheses \ref{CTYas} and \ref{CRCas} are, in general, necessary for existence. We close the section with examples of measures satisfying these hypotheses.

\begin{proposition}
\label{lsc of functional}
Suppose that \ref{CTYas} holds. Then $\mathcal{F}(\sigma) := W_p^p(\nu_\sigma, \nu)$ is lower semicontinuous in the narrow topology.
\end{proposition}

\begin{proof}
For each $\varphi \in C_b(\mathcal{Y})$ the map $x \mapsto \int_{\mathcal{Y}} \varphi(y) d\rho_{x}(y)$ is continuous by \ref{CTYas} and trivially bounded. It follows for any narrowly convergent sequence $\sigma_n \rightarrow \sigma_0$ that
\begin{align*}
\int_{\mathcal{X}} \left( \int_{\mathcal{Y}} \varphi(y) \,d\rho_{x}(y) \right) \,d\sigma_{n}(x) \xrightarrow[]{n \rightarrow \infty} \int_{\mathcal{X}} \left( \int_{\mathcal{Y}} \varphi(y) \,d\rho_{x}(y) \right) \,d\sigma_0(x),
\end{align*}
i.e. $\nu_{\sigma_n} \rightarrow \nu_{\sigma_0}$ narrowly. Thus, $\sigma \mapsto \nu_{\sigma}$ is continuous in the narrow topology. 
Since $W_p(\cdot, \nu)$ is lower semi-continuous in the narrow topology  \cite[Remark 6.12]{Vi09}, so is $\mathcal{F}$. 
\end{proof}

\begin{proposition}
\label{compactness of constraint set}
Suppose that \ref{CTYas} and \ref{CRCas} hold 
and \eqref{minimization problem} is feasible. Then every minimizing sequence of \eqref{minimization problem} has a narrowly convergent subsequence.
\end{proposition}

\begin{proof}
Let $\{\sigma_n\}_{n \in \mathbb{N}}$ be a minimizing sequence,  $
\lim_{n \rightarrow \infty} \mathcal{F}(\sigma_n) = \inf_{\sigma \in \mathcal{P}(\mathcal{X})} \mathcal{F}(\sigma).$ The feasibility of \eqref{minimization problem} ensures $\mathcal{F}(\sigma_n)$ is bounded. In view of Prokhorov's Theorem \cite[Theorem 5.1.3 and Remark 5.1.5]{Am08} it suffices to exhibit some $G: \mathcal{X} \rightarrow [0,\infty]$ whose sublevels $\{x \in \mathcal{X}: G(x) \leq c\}$ are compact and 
$$
\sup_{n \in \mathbb{N}} \int_{\mathcal{X}} G(x) d\sigma_n(x) < \infty.
$$
Let $G(x) = M_p(\rho_{x})$, the sublevels of which are indeed totally bounded by \ref{CRCas} and closed by \ref{CTYas} and the lower semicontinuity of $M_p$ in the narrow topology. By the triangle inequality for $W_p$ and the fact that $M_p(\cdot) = W_p^p(\cdot, \delta_0)$, one has
\begin{align*}
\int_{\mathcal{X}} G(x) d\sigma_n(x) =& \int_{\mathcal{X}} \int_{\mathcal{Y}} |y|^p d\rho_{x}(y) d\sigma_n(x) 
= M_p(\nu_{\sigma_n}) 
= W_p^p(\nu_{\sigma_n}, \delta_0) \\
\leq& 2^{p-1} W_p^p(\nu_{\sigma_n}, \nu) + 2^{p-1} W_p^p(\nu, \delta_0) 
= 2^{p-1} \mathcal{F}(\sigma_n) + 2^{p-1} M_p(\nu).
\end{align*}
This gives the result.
\end{proof}

 Theorem \ref{thm:existence} is now an immediate consequence of  the preceding propositions.

\begin{proof}[Proof of Theorem \ref{thm:existence}]
If \eqref{minimization problem} is not feasible, any $\sigma$ is a minimizer. If \eqref{minimization problem} is   feasible, the result follows   from Propositions \ref{lsc of functional} and \ref{compactness of constraint set}, which ensure that a minimizing sequence converges to a minimizer.
\end{proof}

The following proposition shows that, in general, none of the assumptions in Theorem \ref{thm:existence} may be removed.
\begin{proposition}
\label{necessity of existence hypotheses examples} Let $\mathcal{X} = \mathbb{R}^{d'}$.
\begin{enumerate}[label=(\roman*)]
\item There exist $\nu$, $\rho$ satisfying all hypotheses of Theorem \ref{thm:existence}, except \ref{CTYas}, and for which no minimizer of \eqref{minimization problem} exists.

\item There exist $\nu$, $\rho$ satisfying all hypotheses of Theorem \ref{thm:existence}, except \ref{CRCas}, and for which no minimizer of \eqref{minimization problem} exists.

\end{enumerate}
\end{proposition}
\begin{proof} \
\begin{enumerate}[label=(\roman*)]
\item Let $x_0 \neq 0$, $\nu = \delta_{x_0}$, and $\rho_{x} = \delta_{\phi(x)}$ where $\phi(x) = x$ for $x \neq x_0$ and $\phi(x_0) = -x_0$. Nonexistence of minimizers is trivial by the translation property of the $W_p$ metric. Namely, for any sequence $x_n \rightarrow x_0$ one has $\mathcal{F}(\delta_{x_n}) = |x_n-x_0|^p \rightarrow 0$. However $\nu \neq \nu_{\sigma}$ for every $\sigma$, so $\mathcal{F}(\sigma)>0 \ \forall \sigma$.
\item  
Let $x_0 \neq 0$, $\nu = \delta_0$, and
$\rho_{x} = \delta_{\phi(x)}$ where $\phi(x) = x_0$ for $|x| \leq 1$ and $\phi(x) = x_0 /|x|$ for $|x| > 1$. For any sequence $|x_n| \rightarrow \infty$,  $\mathcal{F}(\delta_{x_n}) = ( |x_0|/|x_n|)^p \rightarrow 0$. However $\nu \neq \nu_{\sigma}$ for every $\sigma$.
\end{enumerate}
\end{proof}

We now provide examples of noise models $\rho$ that satisfy the assumptions \ref{CTYas} and \ref{CRCas}. Here and throughout, for $a \in \mathbb{R}^d$, we define $t_a:\mathbb{R}^d \rightarrow \mathbb{R}^d:y \mapsto y+a$.

\begin{proposition}
\label{(CTY1), (CTY2), and (CRC) sufficient conditions}
Suppose $\rho_{x} = t_{\phi(x)} \# h$ for   $\phi:\mathcal{X} \rightarrow \mathcal{Y}$ measurable, $h \in \mathcal{P}(\mathcal{Y})$. 
\begin{enumerate}[label=(\roman*)]
\item If $\phi$ is continuous, then \ref{CTYas} holds. 
\item If $|\phi|$ has totally bounded sublevel sets,
then \ref{CRCas} holds.
\end{enumerate}
\end{proposition}
\begin{proof} \ 
\begin{enumerate}[label=(\roman*)]
\item Let $x_n \rightarrow x_0$. For any $\varphi \in C_b(\mathcal{Y})$, by the Dominated Convergence Theorem,
\begin{align*}
\left| \int  \varphi \, d\rho_{x_n} - \int  \varphi d\rho_{x_0}  \right|
=& \left| \int  \varphi(\cdot + \phi(x_n)) \, dh  - \int \varphi(\cdot + \phi(x_0)) \, dh  \right| \\
\leq& \int  |\varphi(\cdot + \phi(x_n)) - \varphi(\cdot + \phi(x_0))| \, dh \to 0 .
\end{align*}

\item Let $r>0$ be such that $h(B_r(0)) > 0$ and fix $R \geq r$. For any $x \in \mathcal{X}$ satisfying $|\phi(x)| > R$, it holds
$$
M_p(\rho_{x}) = \int_\mathcal{Y} |y+\phi(x)|^p \,dh(y) \geq (R-r)^p h( B_r(0)).
$$
That is, the $(R-r)^p h(B_r(0))$-sublevel of $x \mapsto M_p(\rho_x)$ is contained in the $R$-sublevel of $|\phi|$. Since $R \geq r$ was arbitrary, this gives the result.
\end{enumerate}
\end{proof}

\subsection{Uniqueness of minimizers}
\label{subsection-uniqueness} \ \\
We begin by proving Theorem \ref{thm:uniqueness} on the uniqueness of minimizers of \eqref{minimization problem}, leveraging Brenier's theorem to prove strict convexity of the objective functional $\mathcal{F}$ along linear interpolations between minimizers. We then give examples that demonstrate the sharpness of our hypotheses and provide examples of noise models that satisfy \ref{INJas}.

We now prove Theorem \ref{thm:uniqueness}.

\begin{proof}[Proof of Theorem \ref{thm:uniqueness}]
The forward implication is trivial, so we show the converse. Assume, for the sake of contradiction that $\sigma_0, \sigma_1 \in \mathscr{A}$, $\sigma_0 \neq \sigma_1$; see equation (\ref{Adef}). By \ref{INJas}, we have $\nu_{\sigma_0} \neq \nu_{\sigma_1}$. By the strict convexity of $W_p^p$ in each argument along linear interpolations  (see e.g. \cite[Proposition 7.19]{santambrogio2015optimal}) and the linearity of $\sigma \mapsto \nu_\sigma$, defining $\sigma_\alpha = (1-\alpha) \sigma_0 + \alpha \sigma_1$, we have  for all $\alpha \in (0,1)$,
\[ W_p^p(\nu_{\sigma_\alpha}, \nu) < (1-\alpha) W_p^p(\nu_{\sigma_0}, \nu)+ \alpha W_p^p(\nu_{\sigma_1}, \nu) = \inf_{\sigma \in \mathcal{P}(\mathcal{X})} W_p^p(\nu_\sigma, \nu),\]
which is a contradiction.
\end{proof}

It is clear that in Theorem \ref{thm:uniqueness} the condition \ref{INJas} may not be removed, and the following proposition shows that, in general, the same is true of the conditions $\nu \ll \mathcal{L}^{d}$ and $p>1$; see Figure \ref{fig:well posedness summary for general denoising problem}. 
Here and throughout, let $\mathcal{N}(\alpha,  \Sigma)$ denote the Gaussian function with mean $\alpha \in \mathbb{R}^d$ and positive definite covariance matrix $\Sigma$,
\begin{align} \label{gaussianfunction}\mathcal{N}(\alpha, \Sigma)(y) = (2\pi)^{-d/2} (\text{det} \Sigma)^{-1/2} \exp(-(y-\alpha)^T \Sigma^{-1} (y-\alpha)/2) .
\end{align} 
For $\beta > 0$ we write $\mathcal{N}(\alpha, \beta)$ as shorthand for $\mathcal{N}(\alpha, \beta I)$.

\begin{proposition} Suppose $\mathcal{X} = \mathbb{R}^{d'}$.
\label{examples showing necessity of uniqueness hypotheses} \ 
\begin{enumerate}[label=(\roman*)]
\item Suppose $p=2$. There exist $\nu$ and $\rho$ satisfying all hypotheses of Theorems \ref{thm:existence} and   \ref{thm:uniqueness}, except $\nu \ll \mathcal{L}^{d}$, and for which minimizers of \eqref{minimization problem} are not unique.
\item Suppose $p=1$. There exist $\nu$ and $\rho$ satisfying all hypotheses of Theorems \ref{thm:existence} and   \ref{thm:uniqueness} and for which minimizers of \eqref{minimization problem} are not unique.
\end{enumerate}
\end{proposition}

\begin{proof}
\ 
\begin{enumerate}[label=(\roman*)]
\item Set $\nu = \delta_0$. Let $q_0,q_1 \in B_1(0)$, $|q_0| \neq |q_1|$ and let $d\rho_{x} = \mathcal{N}(x, v(x)) d\mathcal{L}^{d} $ with
$$
v(x) :=
\begin{cases}
|x-q_0| \cdot |x-q_1| +1 - |x|^2, & |x| \leq 1, \\
\left| \frac{x}{|x|} - q_0\right| \cdot \left| \frac{x}{|x|} - q_1\right|, & |x| > 1.
\end{cases}
$$
The formula for the second moment of a Gaussian asserts
$$
M_2(\rho_{x}) = v(x) + |x|^2 = 
\begin{cases} 
|x - q_0| \cdot |x - q_1| + 1, & |x| \leq 1 ,\\
\left| \frac{x}{|x|} - q_0 \right| \cdot \left| \frac{x}{|x|} - q_1 \right| + |x|^2, & |x| > 1.
\end{cases}
$$
\ref{CTYas} and \ref{CRCas} hold, so by Theorem \ref{thm:existence} a minimizer to \eqref{minimization problem} exists. Notice
$$
\mathcal{F}(\sigma) = M_2(\nu_{\sigma}) = \int_{\mathcal{X}} \int_\mathcal{Y} |y|^2 d\rho_{x}(y) d\sigma(x)
$$
so any $\sigma$ that minimizes the above expression will have support on the point(s) $x_0 \in \mathcal{X}$ that minimize(s)
$
\int_\mathcal{Y} |y|^2 d\rho_{x_0}(y) = M_2(\rho_{x_0}),
$
if any exist. Evidently
$$
\argmin_{x_0 \in \mathcal{X}} M_2(\rho_{x_0}) = \{q_0,q_1\}.
$$
Thus the minimizers of \eqref{minimization problem} are those probability measures with support on $\{q_0,q_1\}$, i.e. the measures of the form $\sigma = (1-\lambda) \delta_{q_0} + \lambda \delta_{q_1}$, $\lambda \in [0,1]$. Clearly \ref{INJas} holds, completing the proof.
\item Let 
$d\nu = \mathbf{1}_{[0,1]^{d}} d\mathcal{L}^{d}$. Let $q_0, q_1 \in \mathbb{R}^{d'}$, $q_0 \neq q_1$.  Set 
$d\rho_{q_0} = \mathbf{1}_{[1,2] \times [0,1]^{d-1}} \mathcal{L}^{d}$ and 
$\rho_{q_1} = \delta_{3/2} \otimes \mu$, $d\mu = \mathbf{1}_{[0,1]^{d-1}} d\mathcal{L}^{d-1}$. For $\alpha \in (0,1)$ let $q_\alpha = (1-\alpha) q_0 + \alpha q_1$ and $\rho_{q_\alpha} = t_{(1-\alpha)\alpha e_1} \# [(1-\alpha) \rho_{q_0} + \alpha \rho_{q_1}]$, where $e_1$ is the unit vector in the first coordinate direction. 
Set $Q := \{q_\alpha\}_{\alpha \in [0,1]}$ and define the projection $\Pi_Q(x) := \arg\min_{q \in Q} |x-q|$ and the distance function $d(x,Q) := |x-\Pi_Q(x)|$. 
For $x \not\in Q$, set $\rho_x = t_{d(x,Q) e_1} \# \rho_{\Pi_Q(x)}$.
One can see that \ref{CTYas} and \ref{CRCas} hold, so that by Theorem \ref{thm:existence} a minimizer to \eqref{minimization problem} exists.  It is evident that minimizers are among the measures supported on $\{q_0, q_1\}$. Furthermore, \ref{INJas} holds by construction.
Clearly, when $p=1$, we have $\mathcal{F}(\delta_{q_0}) = \mathcal{F}(\delta_{q_1}) = 1$. Let $\sigma_{\alpha} = (1-\alpha) \delta_{q_0} + \alpha \delta_{q_1}$ where $\alpha \in [0,1]$. By the suboptimality of $y \mapsto y \cdot e_1$ among $1$-Lipschitz potentials in the Kantorovich-Rubenstein duality for $W_1$ \cite[Proposition 3.1]{santambrogio2015optimal}, we have
$$
W_1(\nu_{\sigma_\alpha}, \nu) \geq \int_{\mathcal{Y}} y \cdot e_1 d\nu_{\sigma_\alpha}(y) - \int_{\mathcal{Y}} y \cdot e_1 d\nu(y) = 1.
$$
Together with the convexity of $W_1$ along linear interpolations and the linearity of $\sigma \mapsto \nu_{\sigma}$, this shows all $\sigma_{\alpha}$ are minimizers of \eqref{minimization problem}.
\end{enumerate}
\end{proof}

The following proposition gives examples demonstrating the general necessity of the hypothesis $\nu \ll \mathcal{L}^{d}$ from Theorem \ref{thm:uniqueness}, including when the noise model is of the special form of transport map noise, $\rho_x = \delta_{\phi(x)}$. 
While  minimizers exist for both examples given in Proposition \ref{examples showing necessity of uniqueness hypotheses for transport map noise}, it is only in the case $d \geq 2$ that we construct such an example which additionally satisfies assumptions \ref{CTYas} and \ref{CRCas}. 
\begin{proposition}
\label{examples showing necessity of uniqueness hypotheses for transport map noise}
Suppose $\mathcal{X} = \mathbb{R}^{d'}$.
\begin{enumerate}[label=(\roman*)]
\item There exist $\nu$ and transport map noise $\rho$ satisfying all hypotheses of Theorem \ref{thm:uniqueness}, except $\nu \ll \mathcal{L}^{d}$, and for which minimizers of \eqref{minimization problem} are not unique.
\item Suppose $d \geq 2$. Then there exist $\nu$ and transport map noise $\rho$ satisfying all hypotheses of Theorem \ref{thm:existence} and Theorem \ref{thm:uniqueness}, except $\nu \ll \mathcal{L}^{d}$, and for which minimizers of \eqref{minimization problem} are not unique.
\end{enumerate}
\end{proposition}
\begin{proof} \ 
\begin{enumerate}[label=(\roman*)]
\item
Set $\nu = \delta_0$. Let $q_0, q_1 \in \mathcal{X}$ with $|q_0| = |q_1|$ and $q_0 \neq q_1$. Let $\rho_{x} = \delta_{\phi(x)}$ where $\phi(q_i) = q_i$ and $|\phi(x)| \geq 2|q_0|$ for $x \neq q_0, q_1$.
Clearly \ref{INJas} holds and the minimizers are the measures supported on $\{q_0, q_1\}$, i.e. those of the form $\sigma = (1-\alpha) \delta_{q_0} + \alpha \delta_{q_1}, \alpha \in [0,1]$.
\item Set $\nu = \delta_0$. Let $q_0, q_1 \in \mathcal{X}$, $q_0 \neq q_1$. Let $\rho_x = \delta_{\phi(x)}$ where $\phi:\mathcal{X} \rightarrow \mathcal{Y}$ is so that $|\phi(q_0)| = |\phi(q_1)|$ and $\phi(q_0) \neq \phi(q_1)$. Let $\gamma:[0,1] \rightarrow \mathcal{Y}$ be a continuous injective curve with $\gamma(0) = \phi(q_0)$, $\gamma(1) = \phi(q_1)$ and $|\gamma| \equiv \text{const}$. For $\alpha \in (0,1)$, set $q_\alpha = (1-\alpha) q_0 + \alpha q_1$. Define $\phi(q_\alpha) = (1+(1-\alpha)\alpha) \gamma(\alpha)$. For $x \notin \{q_\alpha\}_{\alpha \in [0,1]} =: Q$, set $\phi(x) = (1+d(x,Q)) \phi(\Pi_Q(x))$, where $\Pi_Q(x) := \arg\min_{q \in Q} |x - q|$.
Notice $\phi$ is continuous, $|\phi|$ has (totally) bounded sublevels, and $\arg\min|\phi| = \{q_0, q_1\}$. By Proposition \ref{(CTY1), (CTY2), and (CRC) sufficient conditions},  $\rho$ satisfies \ref{CTYas} and \ref{CRCas}. By construction, \ref{INJas} holds. Furthermore, minimizers are not unique.
\end{enumerate}
\end{proof}



We now provide examples of noise models $\rho$ that satisfy the assumption \ref{INJas}.

\begin{proposition}
\label{(INJ) sufficient conditions}
$\sigma \mapsto \nu_{\sigma}$ is injective on $\mathcal{P}(\mathcal{X})$, and in particular \ref{INJas} holds, if any of the following is true:
\begin{enumerate}[label=(\roman*)]
\item $\mathcal{X} = \mathbb{R}^{d'}$, $d = d'$ and $\rho_{x} = t_{x} \# h$ where $h \in \mathcal{P}(\mathcal{Y})$ has nowhere vanishing Fourier transform, e.g., $h$ is a Dirac delta or has Gaussian density w.r.t. $\mathcal{L}^{d}$.
\item $\mathcal{X} = \mathbb{R}^{d'}$, $d = d' = 1$ and $\rho_{x} = t_{x} \# h$, where $h \in \mathcal{P}(\mathcal{Y})$ has compact support. 
 
\item $\rho_{x} = \delta_{\phi(x)}$ and $\phi:\mathcal{X} \rightarrow \mathcal{Y}$ is injective.
\end{enumerate}
\end{proposition}
\begin{proof} \ 
\begin{enumerate}[label=(\roman*)]
\item Because $\rho_{x} = t_{x} \# h$ it follows $\nu_{\sigma} = h \ast \sigma$. The Fourier transform of the convolution of Borel probability measures is the product of the Fourier transform of constituents: $\widehat{h \ast \sigma} = \hat{h} \cdot \hat{\sigma}$. If $\hat{h}$ is nowhere vanishing then the fact that the Fourier transform determines Borel probability measures ensures that $\sigma \mapsto \nu_\sigma$ is injective on $\mathcal{P}(\mathcal{X})$.  
\item If $h$ has compact support then, by Schwartz's Paley-Wiener Theorem, $\hat{h}$ is entire (as an extension to $\mathbb{C}^1$ by the Fourier-Laplace transform), hence nonvanishing a.e. (since otherwise $\hat{h} \equiv 0 \Rightarrow h = 0$, contradicting $h \in \mathcal{P}(\mathcal{Y})$). In $d=1$, the Fourier inversion theorem for $h \in \mathcal{P}(\mathbb{R})$ reads
$$
h((a,b)) + \frac{1}{2} h(\{a,b\}) = \lim_{T \rightarrow \infty} \int_{-T}^T \int_a^b \hat{h}(t) e^{-iyt} \, dy \,dt,
$$
which is unique on equivalence classes of $\hat{h}$, up to a.e. equality. Hence $\hat{h}$ may be taken to be nowhere vanishing, which by (i) implies the desired conclusion.
\item Suppose $\nu_{\sigma_1} = \nu_{\sigma_2}$, that is, $\phi \# \sigma_1 = \phi \# \sigma_2$. Since $\phi$ is an injective, measurable function, there is a left inverse $\phi^{-1}$ that is also Borel measurable \cite[Corollary 15.2]{kechris2012classical}. Pushing forward both sides of the equation by   $\phi^{-1}$ shows $\sigma_1 = \sigma_2$.
\end{enumerate}
\end{proof}

In the following proposition we restrict our view to $d = 1$. While for general noise models, minimizers of \eqref{minimization problem} are not unique when $p=1$, we show that in the case of transport map noise, uniqueness does hold, provided $\phi$ is injective.  This result holds not only when $\nu \ll \mathcal{L}^{d}$ but, more generally, when $\nu$ is atomless. The argument strongly leverages the one dimensional setting, via the characterization of $W_1$ as the $L^1$ distance between CDFs.

\begin{proposition}
\label{d=1}
Suppose $p=1$, $d=1$, $\nu$ is atomless, and $\phi$ is injective. Then minimizers of \eqref{Li et al denoising problem} are unique if and only if \eqref{Li et al denoising problem} is feasible.
\end{proposition}
\begin{proof}
Let $\sigma_0, \sigma_1 \in \mathcal{P}(\mathcal{X})$ be minimizers of \eqref{Li et al denoising problem}. Let $\sigma_{1/2} = \tfrac{1}{2}\sigma_0 + \tfrac{1}{2} \sigma_1$. Given $\rho \in \mathcal{P}(\mathbb{R})$, denote the cumulative distribution function  by $F_\rho(y) = \rho((-\infty, y])$. Since $\nu$ is atomless, $F_{\nu}$ is continuous. By \cite[Proposition 2.17]{santambrogio2015optimal}, we may write
\begin{align*}
W_1(\phi \# \sigma_{1/2}, \nu) 
=& \int_{\mathcal{Y}} |F_{\phi \# \sigma_{1/2}} - F_{\nu}| d\mathcal{L}^1 
= \int_{\mathcal{Y}} \left| \frac{1}{2}(F_{\phi \# \sigma_0} - F_{\nu}) + \frac{1}{2} (F_{\phi \# \sigma_1} - F_{\nu}) \right| d\mathcal{L}^1 \\
\leq& \int_{\mathcal{Y}} \frac{1}{2} |F_{\phi \# \sigma_0} - F_\nu| +  \frac{1}{2}|F_{\phi \# \sigma_1} - F_\nu| d\mathcal{L}^1 \\
=& \frac{1}{2} W_1(\phi \# \sigma_0, \nu) + \frac{1}{2} W_1(\phi \# \sigma_1, \nu).
\end{align*}
Since $\sigma_0$ and $\sigma_1$ minimize \eqref{Li et al denoising problem}, necessarily equality holds in the above inequaliy, meaning $(F_{\phi \# \sigma_0} - F_\nu)(F_{\phi \# \sigma_1} - F_\nu) \geq 0$, $\mathcal{L}^1$-a.e. Moreover since $F_{\phi \# \sigma_i}$ are right-continuous and  $F_{\nu}$ is continuous, we have $(F_{\phi \# \sigma_0} - F_\nu)(F_{\phi \# \sigma_1} - F_\nu) \geq 0$ everywhere. Denote $\mu_i := \phi \# \sigma_i$ and define
\begin{align} \label{Fdef}
F := 
\begin{cases}
\max_i F_{\mu_i}, & F_{\mu_i} - F_\nu \leq 0 \ \ \ \forall i \\
\min_i F_{\mu_i}, & F_{\mu_i} - F_\nu \geq 0 \ \ \ \forall i
\end{cases}
\end{align}

The majority of our proof will be devoted to showing that $F = F_{\phi  \# \bar{\sigma}}$ for some $\bar{\sigma} \in \mathcal{P}(\mathcal{X})$. Once this fact is established, we can then use the fact that  $\sigma_0, \sigma_1$ are minimizers of \eqref{Li et al denoising problem} to conclude that, for each $i$,
$$
\int_{\mathcal{Y}} |F_{\phi  \# \bar{\sigma}} - F_{\nu}| d\mathcal{L}^1 = W_1(\phi \# \bar{\sigma}, \nu) \geq W_1(\phi \# \sigma_i, \nu) = \int_{\mathcal{Y}} |F_{\mu_i} - F_\nu| d\mathcal{L}^1.
$$
On the other hand, by definition of $F$ in equation (\ref{Fdef}),
 \[ |F_{\phi  \# \bar{\sigma}} - F_{\nu}| = |F - F_{\nu}| = \min_i |F_{\mu_i} - F_{\nu}| . \]
Therefore, equality must hold and, moreover, $|F_{\mu_0} - F_\nu| = |F_{\mu_1} - F_\nu|$ $\mathcal{L}^1$-a.e.  By right-continuity and the fact that  $(F_{\mu_0} - F_{\nu})(F_{\mu_1} - F_{\nu}) \geq 0$ everywhere, we obtain $F_{\mu_0} - F_\nu = F_{\mu_1}- F_\nu$, hence  $F_{\mu_0} = F_{\mu_1}$. Thus,  $\mu_0 = \mu_1$. Since $\phi$ is an injective, measurable function, there is a left inverse $\phi^{-1}$ that is also Borel measurable \cite[Corollary 15.2]{kechris2012classical}. 
Pushing forward by $\phi^{-1}$ on both sides of the identity $\mu_0 = \mu_1$, we obtain $\sigma_0 = \phi^{-1} \# \mu_0 = \phi^{-1} \#\mu_1 = \sigma_1$. Therefore,  minimizers are unique. 

It remains to show that $F = F_{\phi  \# \bar{\sigma}}$ for some $\bar{\sigma} \in \mathcal{P}(\mathcal{X})$. We begin by showing that $F = F_{\mu}$ for some $\mu \in \mathcal{P}(\mathcal{Y})$. Since any function that is nondecreasing, right-continuous, and has limits $0$ and $1$ at $-\infty$ and $\infty$, respectively, is the   CDF of a unique Borel probability measure, it suffices to show  that $F$ has these properties.

To see that $F$ is nondecreasing, suppose $a \leq b \in \mathcal{Y}$. Denote 
\[ A := \{F_{\mu_i} - F_\nu \leq 0 \ \forall i\} \text{ and }  B := \{F_{\mu_i} - F_\nu \geq 0 \ \forall i\}. \] If $a,b \in A$ then $F(a) \leq F(b)$ because the inequality $F_{\mu_i}(a) \leq F_{\mu_i}(b)$ is preserved by taking maximums in $i$. Symmetrically, $F(a) \leq F(b)$ if $a,b \in B$. If instead $a \in A, b \in B$ then for each $i$ it holds $F_{\mu_i}(a) \leq F_{\nu}(a) \leq F_{\nu}(b) \leq F_{\mu_i}(b)$ so in fact for all $i,j$ we have $F_{\mu_i}(a) \leq F_{\mu_j}(b)$. Therefore $\max_i F_{\mu_i}(a) \leq \min_i F_{\mu_i}(b)$, i.e. $F(a) \leq F(b)$. Lastly if $a \in B, b \in A$ then the inequality $F_{\mu_i}(a) \leq F_{\mu_i}(b)$ persists under minimization, yielding again $F(a) \leq F(b)$. Hence $F$ is nondecreasing.

To see that $F$ is right-continuous, let $y_0 \in \mathcal{Y}$. If $(F_{\mu_i} - F_{\nu})(y_0) < 0$ for some $i$ then by right-continuity of $F_{\mu_i}$ and continuity of $F_\nu$ it follows $F_{\mu_i} - F_{\nu} < 0$ in some right-neighborhood of $y_0$. In this right-neighborhood $F = \max_i F_{\mu_i}$. The maximum of two right-continuous functions is right-continuous, meaning $F$ is right-continuous at $y_0$.
Symmetrically if $(F_{\mu_i}-F_{\nu})(y_0) > 0$ for some $i$ then $F$ is right-continuous at $y_0$. If, alternatively, $(F_{\mu_i} - F_{\nu})(y_0) = 0 \ \forall i$, then either side of the inequality $\min_i F_{\mu_i} \leq F \leq \max_i F_{\mu_i}$ tends to the common value of $F_{\nu}, F_{\mu_0}, F_{\mu_1}$ at $y_0$ along any sequence approaching $y_0$ from the right. Hence $F$ is right-continuous.

It is clear that $F$ has limits $0$ and $1$ at $-\infty$ and $\infty$, respectively, completing the proof that there exists some $\mu \in \mathcal{P}(\mathcal{Y})$ for which $F = F_{\mu}$.  We will now show that $\mu = \phi \# \bar{\sigma}$ for some $\bar{\sigma} \in \mathcal{P}(\mathcal{X})$.

Towards this we first prove that 
\begin{align} \label{mudominated}
\mu \leq \mu_0 + \mu_1.
\end{align} 
It suffices to show this holds when evaluated at any interval $(a,b]$, $a < b$: that is,
\begin{align} \label{mustshowforF}
F(b) - F(a) \leq (F_{\mu_0}(b) - F_{\mu_0}(a))  + (F_{\mu_1}(b) - F_{\mu_1}(a)) \ , \quad \forall a <b.
\end{align}
If $a,b \in A$, then the fact that $F_{\mu_i}$ are nondecreasing ensures
$$
\max_i F_{\mu_i}(b) - \max_i F_{\mu_i}(a) \leq F_{\mu_0}(b) - F_{\mu_0}(a) + F_{\mu_1}(b) - F_{\mu_1}(a) ,
$$
hence inequality (\ref{mustshowforF}) holds. Similarly, if $a,b \in B$ then, replacing $\max$ with $\min$, we again obtain inequality (\ref{mustshowforF}). If $a \in A, b \in B$ then the trivial bound 
$$
\min_i F_{\mu_i}(b) - \max_i F_{\mu_i}(a) \leq \max_i F_{\mu_i}(b) - \max_i F_{\mu_i}(a)
$$
allows us to once again deduce inequality (\ref{mustshowforF}). Finally, suppose $a \in B$ and $b \in A$. Let $c := \inf \{A \cap [a,b]\}$. By right-continuity of $F_{\mu_i} - F_\nu$, taking a limit along a minimizing sequence yields $c \in A$. That $c \in B$ is trivial if $c=a$ and otherwise follows from $[a,c) \subseteq B$ and the  upper-semicontinuity   of $F_{\mu_i}-F_\nu$ from the left.
In total $c \in A \cap B$, that is, $F_{\mu_0}(c) = F_{\mu_1}(c) = F_{\nu}(c)$.
This means 
$$
\min_i F_{\mu_i}(b) \geq \min_i F_{\mu_i}(c) = \max_i F_{\mu_i}(c) \geq \max_i F_{\mu_i}(a).
$$
By adding $\min_i F_{\mu_i}(b) - \max_i F_{\mu_i}(a) \geq 0$ and rearranging we obtain
\begin{align*}
\max_i F_{\mu_i}(b) - \min_i F_{\mu_i}(a) &\leq \max_i F_{\mu_i}(b) + \min_i F_{\mu_i}(b) - \max_{i} F_{\mu_i}(a) - \min_{i} F_{\mu_i}(a) \\
&= F_{\mu_0}(b) - F_{\mu_0}(a) + F_{\mu_1}(b) - F_{\mu_1}(a),
\end{align*}
which shows (\ref{mustshowforF}).

An immediate consequence of inequality (\ref{mudominated}) that we have just shown is that 
$\text{supp}(\mu) \subseteq \text{supp}(\mu_0) \cup \text{supp}(\mu_1)$. 
Furthermore, we have $\text{supp}(\mu_0) \cup \text{supp}(\mu_1) \subseteq \overline{\text{im} \,\phi}$, for if $y_0 \notin \overline{\text{im}\,\phi}$ then there is some open neighborhood $U$ of $y_0$ which is disjoint from $\overline{\text{im}\,\phi}$ and satisfies $\mu_i(U) = \sigma_i(\phi^{-1}(U)) = \sigma_i(\emptyset) = 0$.
Finally, since $\mu_i(\overline{\text{im}\,\phi} \setminus \text{im}\,\phi) \leq (\phi \# \sigma_i)((\text{im} \,\phi)^c) = 0$, inequality (\ref{mudominated}) implies that $\mu(\overline{\text{im}\,\phi} \setminus \text{im} \,\phi) = 0$. (Note that $\text{im}\,\phi$,  and moreover the image under $\phi$ of any Borel set,
is Borel by the injectivity of $\phi$ and the Lusin-Souslin Theorem \cite[Theorem 15.1]{kechris2012classical}.) 
 
Define now $\bar{\sigma}:\mathscr{B}(\mathcal{X}) \rightarrow [0,1]$ via $\bar{\sigma}(A) = \mu(\phi(A))$. The injectivity of $\phi$ implies $\bar{\sigma}$ is countably additive. Additionally
$$
\bar{\sigma}(\mathcal{X}) = \mu(\text{im}\,\phi) = \mu(\overline{\text{im}\,\phi}) \geq \mu(\text{supp}(\mu_0) \cup \text{supp}(\mu_1)) \geq \mu(\text{supp}(\mu)) = 1,
$$
so indeed $\bar{\sigma} \in \mathcal{P}(\mathcal{X})$. That $\phi \# \bar{\sigma} = \mu$ is immediate: $(\phi \# \bar{\sigma})(B) = \bar{\sigma}(\phi^{-1}(B)) = \mu(\phi(\phi^{-1}(B))) = \mu(B)$ for any Borel subset $B$ of $\mathcal{Y}$. This completes the proof.
\end{proof}

\section{Numerical Method}
\label{section:numerical method}

\subsection{Discretization of the Entropically Regularized Problem}
\label{Entropic regularization}

As described in the introduction, our method for computing approximate minimizers of the unfolding problem \eqref{minimization problem} considers an entropic regularization \eqref{entropic minimization problem} of the original problem. In order to obtain a discrete minimization problem that can be solved numerically, we suppose that all measures are finitely supported. We expect that when these finitely supported measures are reasonable approximations of  continuum counterparts, the minimizers of the discrete  problem will likewise converge to a minimizer of the corresponding continuum problem. Similarly, we expect that, if the entropic regularization is removed slowly enough and the reference measures are well-chosen, minimizers of the entropically regularized problem \eqref{entropic minimization problem} will converge to a minimizer of the original denoising problem \eqref{minimization problem}. However, we leave analysis of these questions to future work. 

We now describe our discretization hypotheses on the measures. First, we suppose that the reference measures $\hat{m}, \hat{\sigma} \in \P(\mathcal{X})$  are empirical,
\begin{align} \label{discreteprior}
\hat{\sigma} = \frac{1}{L} \sum_{k=1}^L \delta_{x_k} \quad \text{ and } \quad \hat{m} = \frac{1}{m} \sum_{i=1}^m \delta_{y_i}.
\end{align}
Next, we suppose that, on the support of $\hat{\sigma}$, the Markov kernel $x \mapsto \rho_x \in \P(\mathcal{Y})$ is finitely supported,
\begin{align} \label{discreteMarkovKernel}
\rho_{x_k} = \sum_{i=1}^m \bR_{ik} \delta_{y_i}
\quad \text{ for } \quad \bR_{ik} \geq 0, \  \sum_{i=1}^m \bR_{ik} = 1, \  \sum_{k=1}^L \bR_{ik} >0 \text{ for all } i, k.
\end{align}
The hypothesis $\sum_{k=1}^m \bR_{ik} >0$ ensures no $y_i$ could be removed from the discretization. Throughout, we use bold font to denote    probability vectors and matrices representing the weights of the measures on their support.  

We suppose further that $\nu$ is finitely supported,
\begin{align} \label{discretemeasureddata}
\nu = \sum_{j=1}^n \bnu_j \delta_{y_j'} ,  \ \text{ for } \bnu_j > 0, \ \sum_{j=1}^n \bnu_j = 1.
\end{align}

These hypotheses cause \eqref{entropic minimization problem} to reduce to a fully discrete optimization problem. The term $\varepsilon {KL}(\sigma |\hat{\sigma})$ in the objective function of \eqref{entropic minimization problem} forces any optimal $\sigma \in \P(\mathcal{X})$ to satisfy   $\sigma \ll \hat{\sigma}$.  Given the discrete form of the Markov kernel (\ref{discreteMarkovKernel}), this ensures that we may restrict our attention to those $\sigma$ and $\nu_\sigma$ which are of the form
\begin{align} \label{sigmanusigmaequation}
\sigma = \sum_{k = 1}^L \bsigma_k  \delta_{x_k} ,  \text{ for } \bsigma_k \geq 0 , \ \sum_{k=1}^L \bsigma_k = 1  \quad \text{ and } \quad 
\nu_{\sigma} = \sum_{i=1}^m (\bR \bsigma)_i \delta_{y_i}.
\end{align}
Similarly, the term $\varepsilon {\rm KL}(\Gamma|\mu_1 \otimes \mu_2)$ in the minimization problem defining the entropically regularized $p$-Wasserstein metric \eqref{entropicWp} forces any optimal $\Gamma \in \P(\mathcal{Y} \times \mathcal{Y})$ to be absolutely continuous w.r.t. $\mu_1 \otimes \mu_2 = \nu_{\sigma} \otimes \nu$, that is
\begin{align*}
\Gamma = \sum_{i = 1}^m \sum_{j = 1}^n \bGamma_{ij} \delta_{(y_i, y_j')} , \text{ for } \bGamma_{ij} \geq 0, \ \sum_{i,j} \bGamma_{ij} = 1.
\end{align*}

Combining these observations and defining the cost matrix  $\bC_{ij} = |y_i - y_j'|^p$, we see that the minimization problem \eqref{entropic minimization problem} is equivalent to the following finite dimensional minimization problem:
\begin{align}
\label{new-entropically-regularized-discrete-denoising-problem}
\argmin_{(\bGamma, \bsigma) \in \Omega} \la \bC , \bGamma \ra
+ \varepsilon \mathbf{KL}(\bGamma | \bR\bsigma \otimes \bnu) 
+ \varepsilon \mathbf{KL}(\bR\bsigma | \tfrac{1}{m} \mathbf{1}_m)
+ \varepsilon \mathbf{KL}(\bsigma|\tfrac{1}{L} \mathbf{1}_L)  ,  \\
 \Omega := \{ (\bGamma, \bsigma):  \bGamma_{ij} , \bsigma_k \geq 0 , \ \bGamma \mathbf{1}_n = \bR\bsigma, \text{ and } \bGamma^T \mathbf{1}_m = \bnu  \} ,\nonumber
\end{align}
where the discrete KL divergence is   $\mathbf{KL}(\bA | \bB) := \sum_{i} \bA_{i} \log \left( \bA_{i}/\bB_{i} \right) - \bA_{i} + \bB_{i}$.
A direct computation shows that \eqref{new-entropically-regularized-discrete-denoising-problem} remains unchanged if the objective function is replaced with
\begin{equation}
\label{entropically-regularized-discrete-denoising-problem-alternative-objective}
\langle \mathbf{C}, \bGamma \rangle - \varepsilon \mathbf{H}(\bsigma) - \varepsilon \mathbf{H}(\bGamma),
\end{equation}
where the discrete entropy is $\mathbf{H}(\mathbf{A}):= -\sum_i \mathbf{A}_i \log(\mathbf{A}_i) - \mathbf{A}_i$.

Analogously to the case of classical entropic regularization  \cite[Proposition 4.3]{Pe19}, there is a unique minimizer of  (\ref{new-entropically-regularized-discrete-denoising-problem}), which can be characterized via lower dimensional scaling variables. The proof is deferred to the Supplementary Material.
\begin{proposition}
\label{Peyre-Cuturi 4.3 analog} 
Denoting $\mathbf{K} = e^{-\bC/\varepsilon}$, the unique minimizer of \eqref{new-entropically-regularized-discrete-denoising-problem} satisfies \begin{align*}
\bGamma_{ij} = \mathbf{u}_i \mathbf{K}_{ij} \mathbf{v}_{j} \text{ and } \bsigma_k = e^{- \sum_{i=1}^m \bR_{ik} \mathbf{f}_i / \varepsilon} \quad \text{ where } \mathbf{u}_i = e^{\mathbf{f}_i/\varepsilon}
\end{align*}
for two (unknown) scaling variables $(\mathbf{u}, \mathbf{v}) \in \mathbb{R}^{m}_+ \times \mathbb{R}^n_+$.
\end{proposition}

\subsection{Computation of approximate minimizer}
\label{CompAppMin}
We compute approximate minimizers of (\ref{new-entropically-regularized-discrete-denoising-problem}) by reformulating the problem as the minimization of a convex function over the intersection of two affine sets. Defining
\begin{align*}
\tilde{\bC} &= \begin{bmatrix} \bC & \boldsymbol{+\infty}_{(m \times 1)} \\ \boldsymbol{+\infty}_{(L \times n)} & \mathbf{0}_{(L \times 1)}\end{bmatrix}  , \quad \bP = \begin{bmatrix} \bGamma & \boldsymbol{0}_{(m \times 1)} \\ \boldsymbol{0}_{(L \times n)} & \bsigma\end{bmatrix}, \\
  \tilde{\bK}&= e^{-\tilde{\bC}/\varepsilon}  , \quad    \bB = \begin{bmatrix} \textbf{Id}_m ; -\bR \end{bmatrix} ,   \quad \bb = [\bnu; \,1],   
\end{align*}
we observe that (\ref{new-entropically-regularized-discrete-denoising-problem}) is equivalent to
\begin{align}
\label{minimization problem discrete entropy regularized in Bregman/Sinkhorn form}  
 &\argmin_{\bP  \in \mathcal{A} \cap \mathcal{B}} \varepsilon \mathbf{KL}(\bP | \tilde{\bK})   \\
\mathcal{A}  =  \{\bP \in \mathbb{R}^{(m+L) \times (n+1)}  : \bP \mathbf{1}_{n+1} &\in \ker \bB\}   , \  \mathcal{B}  =  \{\bP \in \mathbb{R}^{(m+L) \times (n+1)}: \bP^T \mathbf{1}_{m+L} = \bb\} \nonumber
\end{align}

\begin{remark} \label{choiceOfReferenceMeasure}
Note that equation \eqref{entropic minimization problem} is only one of numerous possible definitions of entropic regularization of the unfolding problem \eqref{minimization problem} that lead to the above discrete formulation. For example, modifying the reference measure in classical definition of $W_{p,\varepsilon}$ from  $\mu_1 \otimes \mu_2$ to $\hat{m} \otimes \nu$ eliminates the need for the term $\varepsilon KL(\nu_\sigma|\hat{m})$. One may also choose to allow for greater parametric freedom in \eqref{entropic minimization problem} by changing the regularization weight $\varepsilon$ in the term $\varepsilon KL(\sigma|\hat{\sigma})$, say to $\varepsilon'$. This corresponds to replacing $\bsigma$ in the definition of $\bP$ with $\tfrac{\varepsilon'}{\varepsilon} \bsigma$. In the present work the regularization weights are kept equal for the sake of convenience.
\end{remark}

We compute an approximate minimizer via Bregman projections \cite{bregman1967relaxation},
\begin{align}
\label{Bregman iterations}
 {\bP}^{2l+1} = {\rm Proj}_{\mathcal{A} }^{\mathbf{KL}}( {\bP}^{2l}), \ \ \ 
 {\bP}^{2l+2} = {\rm Proj}_{\mathcal{B} }^{\mathbf{KL}}( {\bP}^{2l+1}), \ \ \ l = 0, \dots, L_{OT} .
\end{align}
where  ${\rm Proj}^{\mathbf{KL}}_{\mathcal{C} }(\mathbf{P}) =   \argmin_{\bP' \in \mathcal{C} } \textbf{KL}(\bP' | \bP)$. The first  projection  forces the first marginal of $\mathbf{P}$ to correspond to weights for a pair $(\bnu_{\bsigma}, \bsigma)$; the second projection forces the second marginal of $\mathbf{P}$ to  match the weights of the measured data $\bnu$.
\begin{remark}[Rate of convergence] \label{rate of convergence}
The argmin problem \eqref{minimization problem discrete entropy regularized in Bregman/Sinkhorn form} is an entropy regularized unbalanced optimal transport problem, and approximating the minimizer via Bregman projections is a prominent approach; see \cite{chizat2026sharper,chizat2018scaling,Pe19} and the references therein. For example, recent work due to Peyr\'e \cite{Pe26}, which applies\footnote{To see the correspondence between \eqref{minimization problem discrete entropy regularized in Bregman/Sinkhorn form} and recent work of Peyr\'e \cite{Pe26}, note that one may replace $\tilde{\bC}$ by a vectorization of $\bC$ appended with $\mathbf{0}_{(L \times 1)}$, denoted by $\bc$, and replace $\bP$ by a vectorization of $\bGamma$ appended with $\bsigma$, denoted by $\bp$. This change of variables shows \eqref{minimization problem discrete entropy regularized in Bregman/Sinkhorn form} is equivalent to minimizing $\varepsilon \mathbf{KL}(\bp |e^{-\bc/\varepsilon})$, subject to affine constraints on $\bp$. Note that, in these new variables, we have that  $e^{-\bc/\varepsilon}$ is strictly positive (which failed for $\tilde{\bK}$). Furthermore, the Bregman projections in the original variables (\ref{Bregman iterations}) are equivalent to Bregman projections in the new variables. } to iterations of the form \eqref{Bregman iterations}, shows convergence to optimum in terms of a dual formulation of the problem, with rate of convergence  $\mathcal{O}(1/l)$. Importantly, Peyr\'e's work develops sufficient conditions under which the constant in the rate of convergence scales merely linearly in $1/\varepsilon$. However, we leave further analysis of the dependence  on $\varepsilon$ for our iterations \eqref{Bregman iterations} to future work.

\end{remark}

As in previous work on entropically regularized transport problems \cite{chizat2018scaling,Pe19}, rather than computing the projections (\ref{Bregman iterations}) directly, which consists of updating a matrix of size $(m+L) \times (n+1)$, we instead  use a  Sinkhorn-type scheme , which only requires us to update two scaling variables of size $m+L$ and $n+1$,
\begin{equation} 
\label{Modified Sinkhorn iterations which are equivalent to Bregman}
 {\bu}^{l+1} = \frac{{\rm Proj}_{ \ker \bB}^{\mathbf{KL}} ( {\bu}^l \odot \tilde{\bK}  {\bv}^l)}{\tilde{\bK}{\bv}^l}, \ \ \  {\bv}^{l} = \frac{\bb}{\tilde{\bK}^T {\bu}^{l}} , \quad l = 0, \dots, L_{OT} .\end{equation}
 The equivalence of  (\ref{Bregman iterations})-(\ref{Modified Sinkhorn iterations which are equivalent to Bregman}) is proved in the Supplementary Material. In particular, we show that $\mathbf{P}^{2l} ={\rm diag}(\mathbf{u}^{l}) \tilde{\mathbf{K}}  {\rm diag}(\mathbf{v}^{l})$, hence each iteration of (\ref{Modified Sinkhorn iterations which are equivalent to Bregman}) yields a current approximation of $\bsigma$ via the formula
\begin{align*}
\bsigma_k = \mathbf{P}^{2l}_{m+k,n+1} = \mathbf{u}^l_{m+k}   \mathbf{v}^l_{n+1} , \quad k = 1, \dots ,L .
\end{align*}
In practice, we implement the iterations (\ref{Modified Sinkhorn iterations which are equivalent to Bregman}) in logarithmic coordinates, to avoid numerical overflows \cite{chizat2018scaling,Pe19}.

Our initialization of  the iterations  (\ref{Modified Sinkhorn iterations which are equivalent to Bregman}) is based on a discrete prior $\bsigma^0$. Given a probability vector $\bsigma^0$ and $\varepsilon_\text{init}>0$, the optimal solution $\bgamma$ to the classical entropy regularized optimal transport problem,
\begin{align}\label{OTproblemforcomputingprior}
\bgamma  := \argmin_{\bgamma \in \mathbb{R}^{m \times n}_+ , \   \bgamma \mathbf{1}_n = \bR\bsigma^0, \ \bgamma^T \mathbf{1}_m = \bnu}\la \bgamma , \bC \ra- \varepsilon_\text{init} H(\bgamma)   
\end{align}
is of the form $\bgamma_{ij} = \bc_i \bK_{ij} \bd_j$ for some $\bc \in \R^m_+$, $\bd \in \R^n_+$ \cite[Proposition 4.3]{Pe19}, where we take $\bK = e^{-\bC/\varepsilon_{\text{init}}}$. The iterations (\ref{Modified Sinkhorn iterations which are equivalent to Bregman}) are then  initialized as
\begin{align} \label{initializationofu}
{\bu}^0 := [\bc ; \bsigma^0]  .
\end{align}
 With this choice,  if our prior information were perfect---that is, if $\bsigma^0$ and the optimal transport plan $\bGamma$ between $\bR \bsigma^0$ and $\bnu$ were the optimizers of our   discrete problem (\ref{new-entropically-regularized-discrete-denoising-problem})---and the entropic regularizations coincide, $\varepsilon_\text{init} = \varepsilon$, then $\bu^l \equiv \bu^0$ and $\bv^l \equiv \bv^0$ for all $l \geq 0$, so that the iterations stay at the optimum. (A proof of this result is provided in the Supplementary Material.) More generally, we expect that if $\bsigma^0$ is ``close'' to optimum, $\bu^0$ and $\bv^0$ constructed in this way are likewise ``close'' to the optimal values reached along the iterations. 
 
 \begin{remark}[Positivity of scaling variables] \label{positivityscaling}
Our    initialization of  $\bu^0$ in equation (\ref{initializationofu}) ensures that the iterations (\ref{Modified Sinkhorn iterations which are equivalent to Bregman}) are always well-defined, in that the denominators are always strictly positive, and likewise ensures that the scaling variables $\bu^{l}$ and $\bv^l$ remain strictly positive along the iterations. 
 \end{remark}

Finally, as the projection operator ${\rm Proj}_{ \ker \bB}^{\mathbf{KL}}$ lacks a closed form analytic expression, in practice, we approximate its value via a  Douglas-Rachford splitting algorithm. Given $\bw \in \mathbb{R}^{m+L}_+$ and a step size $\tau >0$, we initialize $\bx^0 = \bw$ and solve the iterations
\begin{align} \label{DRsplitting}
\begin{cases}
\by^n = (\mathbf{\rm Id} - \bB^T (\bB \bB^T)^{-1} \bB) \bx^n, \\
\bz^n = \tau W\left( \frac{\bw}{\tau} \odot e^{(2 \by^n - \bx^n)/\tau} \right), \\
\bx^{n+1} = \bx^n + \bz^n - \by^n,
\end{cases}
\end{align}
where the matrix operator first step is precomputed using the QR decomposition of $\bB$ and the function $W$ in the second step is the Lambert W-function. It is a classical result that, as $n \to +\infty$, we have $\bx^n \to {\rm Proj}^{\textbf{KL}}_{{\rm ker} \bB}(\bw)$ \cite[Corollary 27.7]{bauschke2011convex}.\footnote{Indeed, setting $f = \mathbf{KL}(\cdot|\bw)$  and $g = \iota_{\ker \mathbf{B}}$, the Douglas-Rachford splitting algorithm converges to a minimizer of $f+g$, since both $f$ and $g$ are proper and lower semicontinuous, the sublevel sets of $f+g$ are compact, and $[\mathbf{R} \mathbf{1}_L  \ \mathbf{1}_L]^T \in \mathbb{R}^{m+L}_{>0} \cap \ker \mathbf{B} =  (\text{ri dom } f) \cap (\text{ri dom } g) $, by equation  (\ref{discreteMarkovKernel}). }
In all our simulations, we obtain good results using 25 iterations of (\ref{DRsplitting}) to approximate ${\rm Proj}_{ \ker \bB}^{\mathbf{KL}}$. For this reason, in our analysis of numerical performance below, we primarily focus on the role of the iterations (\ref{Modified Sinkhorn iterations which are equivalent to Bregman}) in OT unfolding.

   \subsection{Choice of  parameters} In the numerical experiments that follow, we make the following choices of parameters:
$\epsilon = 3\times10^{-5}$, 
$\epsilon_\text{init} = 0.01$, and $\tau  = 0.001$. We use 100 iterations of the classical Sinkhorn algorithm to compute our initialization (\ref{initializationofu}), and we use 25 iterations of the DR method (\ref{DRsplitting}) to approximate ${\rm Proj}_{ \ker \bB}^{\mathbf{KL}}$. We choose $\bsigma^0 = \frac{1}{L} \mathbf{1}_{L}$, so that the prior coincides with the reference measure $\hat{\sigma}$; see equation (\ref{discreteprior}).

This choice of parameters reflects the fact that, in practice, we observe good results and faster convergence choosing $\epsilon_\text{init}$ much larger than $\epsilon$. Since the prior $\bsigma^0$ does not contain perfect information, perfectly accurate  initialization of the scaling variables $\bu^0, \bv^0$ is not necessary, and a larger choices of $\epsilon_\text{init}$ allows us to use fewer iterations of the classical Sinkhorn to compute the prior. Interestingly, we also observe that we are more likely to encounter numerical overflows when $\epsilon_\text{init}$ is too small than when $\epsilon$ is too small, likely due to the fact that both marginals in the optimal transport problem (\ref{OTproblemforcomputingprior}) are fixed, while our unfolding problem (\ref{new-entropically-regularized-discrete-denoising-problem}) only fixes one marginal, the measured data.
   
\section{Numerical Results: Comparison of OT and RL}
\label{section numerical results}
We now compare the performance of our optimal transport (OT) based unfolding method to the classical Expectation-Maximization approach, commonly known as Richardson-Lucy (RL). The code for all numerical experiments is available at \url{https://github.com/katycraig/unfolding-Wasserstein-loss}.

\subsection{Example unfolding problems} \label{unfoldingproblemsetupsection}
We consider the performance of OT and RL unfolding  on two examples of unfolding problems: one dimensional problems, based on discretizations of simple   continuum models, and two dimensional problems, based on a physically motivated  jet mass unfolding problem. In both cases, the problems are constructed so that an exact solution is known, providing a baseline for evaluating the relative performance of OT and RL unfolding.
\subsubsection{Setup of one dimensional unfolding problem}
Given a transport map $t: \R \to \R$ and   $\beta >0$, we consider continuum noise models induced by the Markov kernel
\begin{align} \tilde{\rho}_x = \mathcal{N}(t(x),\beta) d \mathcal{L} ,\label{MarkovKernelSimulations}
\end{align}
where the Gaussian function $\mathcal{N}$ is as defined in   equation (\ref{gaussianfunction}).
Our continuum measured  data $\tilde{\nu}$ is obtained by applying this noise model to   $\tilde{\sigma}_{\text{true}} \in \mathcal{P}(\mathbb{R})$, which is a Gaussian mixture of  the form  
\begin{align} \label{sigmatrue}
\tilde{\sigma}_{\text{true}}= \sum_{i=1}^3 w_i \mathcal{N}(c_i, v_i)(x) d \mathcal{L}(x) \text{ for } c_i \in \R, v_i > 0 . 
\end{align}
In other words, $\tilde{\nu} = (\mathcal{N}(0,\beta) d \mathcal{L})* (t \# \sigma_{\text{true}})$. In what follows, all components of these measures that are Gaussian are  sampled via  \texttt{scipy.stats.norm.rvs}.

Based on this underlying continuum problem, we   construct a fully discrete unfolding problem, as in section \ref{Entropic regularization} as follows. We fix a reference measure $\hat{\sigma} = \frac{1}{L} \sum_{k=1}^L \delta_{x_k}$ and sample $\tilde{\sigma}_\text{true}$ to obtain a discrete approximation $\sigma_{\text{true}} := \frac{1}{L'} \sum_{k=1}^{L'} \delta_{x_k'}$. Next, we sample the Markov kernel $\tilde{\rho}_x$ for all $x$ in the support of $\hat{\sigma}$ and $\sigma_{\text{true}}$,   denoting the resulting samples by $\{z_j(x_k)\}_{j=1}^{M}$  and $\{z_j'(x_k')\}_{j=1}^{M'}$. 

We choose the following parameters in our numerical examples:
\begin{table}[h!]
    \centering
    \begin{tabular}{| c | c| c |}
        \hline
       reference $\hat{\sigma}$ (\ref{discreteprior}) & $  \frac{1}{L}\sum_{k=1}^L \delta_{x_k}$&  $x_k$ uniform   from -1 to 1   \\ \hline
       true signal $\tilde{\sigma}_{\text{true}}$ (\ref{sigmatrue}) &   $\sum_{i=1}^3 w_i \mathcal{N}(c_i, v_i)  d \mathcal{L} $ & $c_1 = \frac34, c_2 = -\frac34, c_3 = 0, $\\
       && $ v_i \equiv \frac{1}{20} , w_1 = w_3 = \frac14, w_2 = \frac12$ \\ \hline
Markov kernel  $\tilde{\rho}_x$  (\ref{MarkovKernelSimulations}) & $\mathcal{N}(t ,\beta) d \mathcal{L}$ & $\beta =\frac{1}{100} , t  = {\rm id} + \frac12 {\rm sgn} $ 
\\ \hline
    \end{tabular}
    \caption{Default parameter choices for numerical examples}
\end{table}

\subsubsection{Setup of two dimensional unfolding problem} \label{2dsetup}
Synthetic data for the physically motivated jet mass unfolding problem are generated using Pythia~\cite{Sjostrand:2007gs} to simulate proton-proton collisions and the resulting collision debris.  Detector effects are emulated with Delphes~\cite{deFavereau:2013fsa}.  To generate data for the noise model, Pythia events are written to HEPMC~\cite{Dobbs:2001ck}, so that the same event can be passed through Delphes many times.  For simplicity, we consider two well-studied observables that also have significant detector distortions: the jet mass and groomed jet mass~\cite{Larkoski:2014wba} of the highest transverse momentum jets in each event. The ``true'' collision debris and noisy measurements are then written to separate files for each of the two observables. The joint measurement of these observables is useful to study short and long distance properties of the strong force.

Based on this data, we construct the two dimensional unfolding problem following the same procedure as in the one dimensional case. We begin by  fixing a reference measure $\hat{\sigma} = \frac{1}{L} \sum_{k=1}^L \delta_{x_k}$, where the first column of each file gives the locations of the prior $ {x}_k = (x^1_k, x^2_k) \in \mathbb{R}^2$ for rows $k=1,\dots, L$. In order to construct our measured data $\nu$, we fix a $\sigma_\text{true} =   \sum_{k=1}^{L'} \bsigma_{k}' \delta_{ {x}_k'}$, again using the first column of each file for the locations of $ {x}_k'$, which are chosen from the subsequent $L'$ rows following those used to construct $\hat{\sigma}$. In order to ensure that $\hat{\sigma}$ is not so close to $\sigma_{\text{true}}$ that the unfolding problem is trivial, rather than choosing the weights $\bsigma_k'$ to be identically $\tfrac{1}{L'}$, they are instead chosen  to be proportional to $\mathcal{N}((30,-5), \text{diag}(10,2))(x_k')$ and normalized so that $\sum_{k=1}^{L'} \bsigma_k' = 1$.    
Next, we use the remaining columns of each file to obtain samples of the Markov kernel, with $\{ {z}_j(x_k)\}_{j=1}^M = \{(z^1_j(x_k),z^2_j(x_k))\}_{j=1}^M$ denoting the samples for each $ {x}_k$ in $\hat{\sigma}$ and $\{z_j'(x_k')\}_{j=1}^{M'}$ denoting the samples for each $ {x}_k'$ in $\sigma_{\text{true}}$. We modify the files by replacing values of $ {z}_j(x_k)$ or $ {z}'_j(x_k)$ that are more than five standard deviations away from mean with the average value among all $j$ and $k$ for each of the first and second components.

\subsubsection{Definition of discrete unfolding problem}
With  $\hat{\sigma} = \frac{1}{L} \sum_{k=1}^L \delta_{ {x}_k}$, ground  truth $\sigma_\text{true} = \sum_{k=1}^{L'} \bsigma_{k}' \delta_{ {x}_k'}$, and samples $\{ {z}_j(x_k)\}_{j=1}^M$ and $\{ z'_j(x_k)\}_{j=1}^{M'}$ of the Markov kernel $\rho_x$ for $x$ in the support of $\hat{\sigma}$ and $\sigma_\text{true}$, as described in the previous sections, we are now able to complete our definition of the discrete unfolding problem, as in section \ref{Entropic regularization}.
We   construct the discrete Markov kernel $\rho_{x_k}$ as in equation (\ref{discreteMarkovKernel}), by  taking $m = LM$ and defining the locations and weights by
 \[   \{y_i\}_{i=(k-1)M+1}^{kM} = \{z_j(x_k)\}_{j=1}^{M} , \quad {\mathbf R}_{ik} = \begin{cases} 1/M &\text{ if }i \in [(k-1)M+1, kM], \\ {\mathbf R}_{ik} = 0 &\text{ otherwise, }\end{cases} \]
 for all $k=1,\dots,L$.  
 Finally, we construct the discrete measured data $\nu$, as in equation (\ref{discretemeasureddata}), by first taking $n= L' M'$,  and  constructing the locations  $\{y_i'\}_{i=1}^{n}$ and weights $\mathbf{R}_{ik}'$ as above. We then define $ \nu := \nu_{\sigma_\text{true}} =  \sum_{i,k} \mathbf{R}_{ik}' \bsigma'_k \delta_{y_i'}.$ 

In order to compare the performance of unfolding with a Wasserstein loss to the classical Richardson-Lucy approach, we also consider binned approximations of the above quantities. We first describe our binning of the Markov kernel. Given $n_\text{bin} \in \mathbb{N}$, we approximate the locations $\{y_i\}_{i=1}^m$ by equally spaced locations $\{\bar{y}_j\}_{j=1}^{n_\text{bin}}$ on a grid between the minimum and maximum in each coordinate. Letting $Q(\bar{y}_j)$ denote the cell of the grid partition centered at $\bar{y}_j$, the weights are defined by $\overline{\textbf{R}}_{jk}  = \sum_{ i:   y_i \in Q(\bar{y}_j) } \textbf{R}_{ik}$. Finally, to ensure that this binning approximation satisfies the absolute continuity hypotheses (\ref{KLnecessities}) with reference measure $m = \frac{1}{n_\text{bin}} \sum_{j=1}^{n_\text{bin}} \delta_{\bar{y}_j}$, we perturb the weights by  $0<\epsilon_\text{bin} \ll 1$. This leads to the binned Markov kernel
$ \bar{\rho}_{x_k} := \sum_{j=1}^{n_\text{bin}}((1-\epsilon_\text{bin}) \overline{\textbf{R}}_{jk} + \epsilon_\text{bin} \mathbf{1}/n_\text{bin})\delta_{\bar{y}_j}$.
We bin the measured data in a similar way, approximating $\nu$ by $\sum_{j=1}^{n_\text{bin}} \bar{\boldsymbol{\nu}}_j \delta_{\bar{y}_j}$, where   $\bar{\boldsymbol{\nu}}_j= \sum_{i: y_i \in Q(\bar{y_j})} \bnu_i$. As with the Markov kernel,  to enforce hypotheses (\ref{KLnecessities}), we perturb the weights by $0<\epsilon_\text{bin} \ll 1$. This leads to the binned measured data $\bar{\nu} := \sum_{j=1}^{n_{\text{bin}}}  ((1-\epsilon_\text{bin}) \bar{\bnu} + \epsilon_\text{bin} \mathbf{1}/n_\text{bin}) \delta_{\bar{y}_j}$. In all of our simulations, we choose $\epsilon_{\text{bin}} = 10^{-40}$.

In the simulations that follow, we evaluate the relative performance of OT and RL unfolding by computing the value of $W_2^2(\nu_\sigma,\nu)$ along iterations. At each iteration of both algorithms, we take the current approximations of the true data, $\sigma_{OT}$ and $\sigma_{RL}$, and apply the \emph{unbinned} discrete noise model to each as in equation (\ref{sigmanusigmaequation}), in order to obtain their images under the noise model, $\nu_{\sigma_{OT}}$ and $\nu_{\sigma_{RL}}$. We then measure the accuracy of the OT and RL algorithms by comparing these images to the true measured data $\nu$, in terms of the square Wasserstein metric. 

Given the fundamental role of the KL divergence in the RL algorithm, the KL divergence between $\nu$ and $\nu_{\sigma_{OT}}$, $\nu_{\sigma_{RL}}$ might also be an interesting point of comparison. However, since we typically do not have   $\nu \ll \nu_{\sigma_{OT}}$, this leads to ${\rm KL}(\nu|\nu_{\sigma_{OT}}) = +\infty$. Analysis of  OT and RL performance under other metrics or statistical divergences is an interesting question that we leave to future work.

\begin{figure}[htbp]
\begin{center}
\includegraphics[scale=0.55, trim={.4cm .55cm .45cm .45cm},clip]{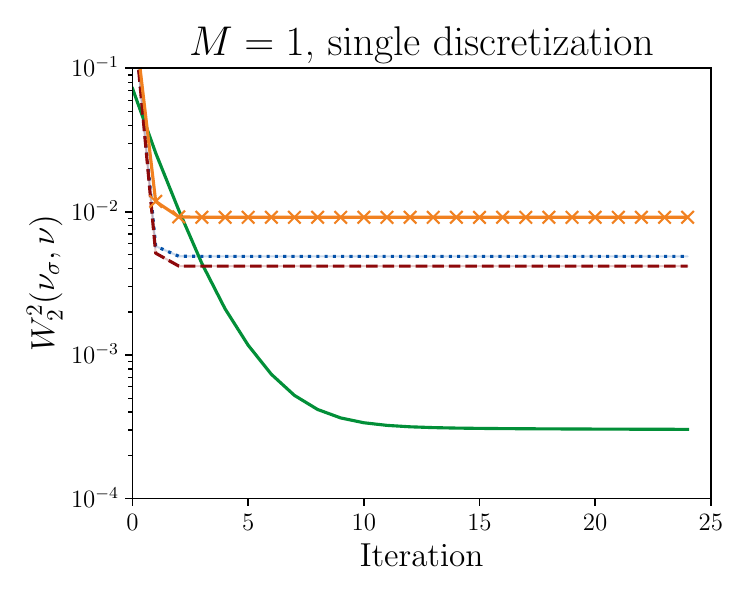}
\includegraphics[scale=0.55, trim={2cm .55cm .45cm .45cm},clip]{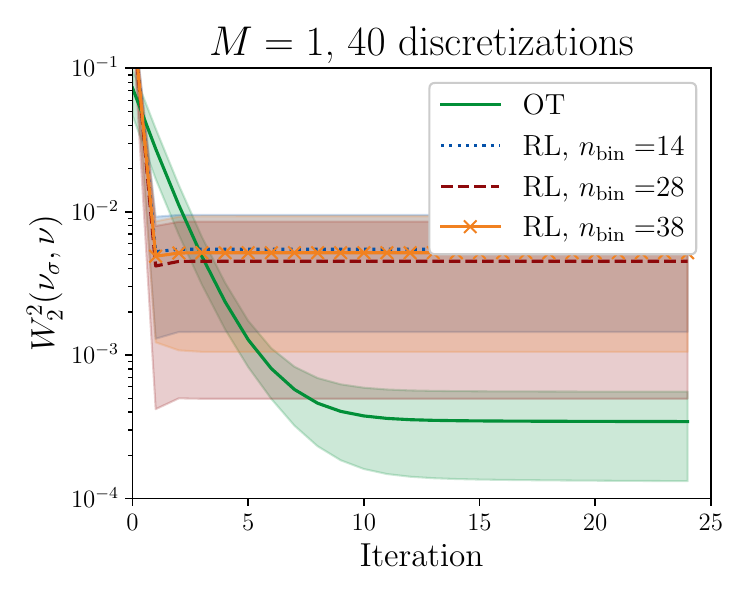}  
\caption{
Behavior of $W_2(\nu_\sigma, \nu)$ along iterations for OT and RL unfolding methods. Left: a single discretization of the continuum one dimensional  unfolding problem. Right: mean and standard deviation of behavior across 40 independent discretizations of the same continuum  unfolding problem. While OT requires more iterations to converge, it achieves higher accuracy.}
\end{center} \label{M1manyseeds}
\end{figure}

\subsection{One dimension}
 We begin by analyzing the role of the Markov kernel $\rho_x$ and  binning parameter $n_\text{bin}$ on the one dimensional unfolding problem described in section \ref{unfoldingproblemsetupsection} with $L= L'=150$ and $M'=3$. As  the performance of the OT method is fairly stable for different choices of Sinkhorn regularization $\epsilon$,  we choose $\epsilon = 3 \times 10^{-5}$ in all simulations. At the end of this section, we demonstrate the effect of different choices of $\epsilon$.

 Figure \ref{M1manyseeds} compares the behavior of OT to RL in terms of the value of $W_2^2(\nu_\sigma,\nu)$ along iterations of the method for $M=1$. 
 The left panel shows the results of a single  discretization obtained by sampling the  continuum problem. In this case, we observe that RL  converges more quickly for all choices of $n_\text{bin}$, while OT achieves the best accuracy in terms of $W_2^2(\nu_\sigma,\nu)$. The performance of RL is non-monotonic in the number of bins, with $n_\text{bin} =28$ achieving the  best accuracy. The right panel illustrates the average behavior of RL and OT across 40 independent discretizations of the continuum problem, with mean performance shown by the solid line and variance shown by the shaded region. While the large width of some shaded regions demonstrates that performance of unfolding methods does change for different discretizations of the same underlying continuum problem, we still observe the same qualitative trends as in the case of a single discretization.
 
\begin{figure}[htbp]
\begin{center}
Measured data $\nu$ vs. unfolding approximations $\nu_{OT}$ and $\nu_{RL}$
\includegraphics[scale = 0.53,trim={.4cm 6.8cm .4cm 7.4cm},clip]{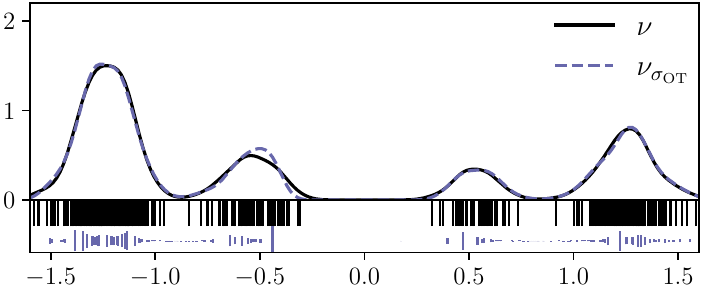}
\includegraphics[scale=0.53,trim={.8cm 6.8cm 0cm 7.4cm},clip]{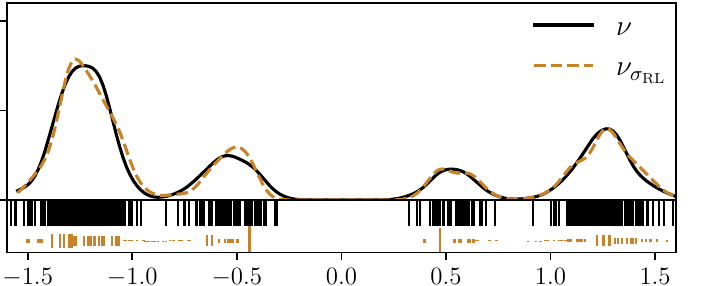}
\caption{Comparison of  measured data $\nu$   to $\nu_\sigma$, where $\sigma$ is from the final iteration of OT unfolding (left panel) vs. the final iteration of RL unfolding with 28 bins (right panel). In agreement with  Figure \ref{M1manyseeds}, we observe that,   OT unfolding achieves better accuracy of the approximation $\nu \approx \nu_\sigma$ compared to RL unfolding on this task.}
\end{center} \label{1dmeasureddata}
\end{figure}

For the same unfolding problem,  Figure \ref{1dmeasureddata} provides a visual comparison of between the measured data $\nu$ and $\nu_\sigma$, where $\sigma$ is from the final iteration of OT unfolding (left panel) vs. the final iteration of RL unfolding with 28 bins (right panel). As all quantities in the numerical scheme are discrete (see section \ref{Entropic regularization}), the black tick marks along the bottom of the figure represent the locations of the true measured data $\{ y_j '\}_{j=1}^n$ and their weights $\{\bnu_j\}_{j=1}^n$. The purple and orange tick marks represent the locations of $\{y_i\}_{i=1}^m$ and their weights $\{\bR\bsigma)_i\}_{i=1}^m$. For ease of visual comparison, the top portion of each panel plots Gaussian kernel density estimations of each empirical measure with bandwidth $0.05$. In agreement with  Figure \ref{M1manyseeds}, we observe that, in this example, OT unfolding achieves better accuracy of the approximation $\nu \approx \nu_\sigma$ compared to RL unfolding.

\begin{figure}[htbp]
\begin{center}
\includegraphics[scale=0.41,trim={.6cm .7cm .8cm 0cm},clip]{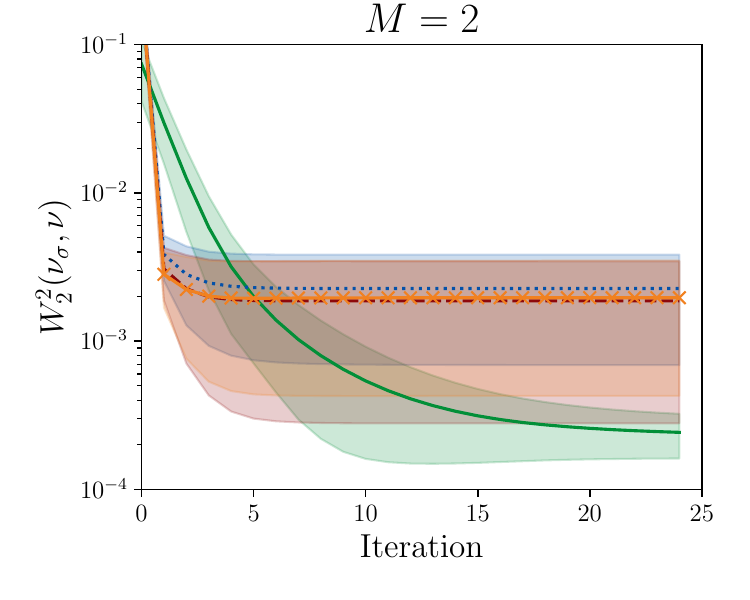}
\includegraphics[scale=0.41,trim={2.2cm .7cm .8cm 0cm},clip]{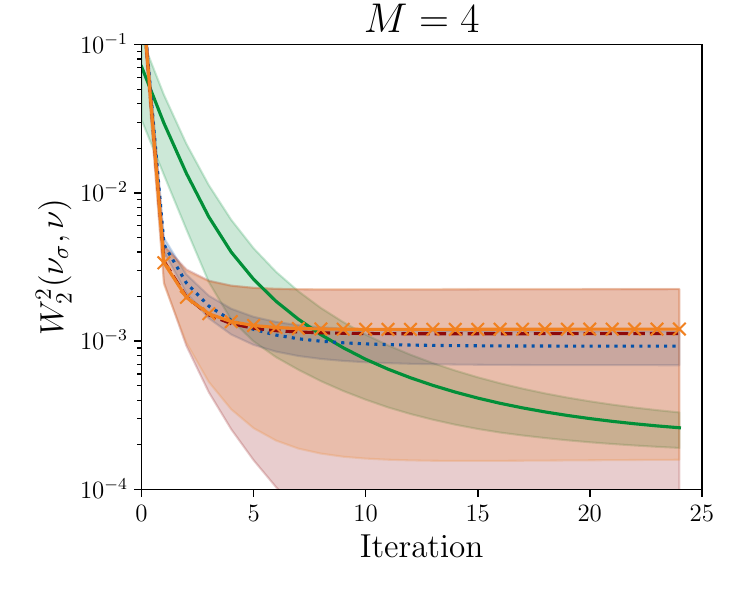}
\includegraphics[scale=0.41,trim={2.2cm .7cm .8cm 0cm},clip]{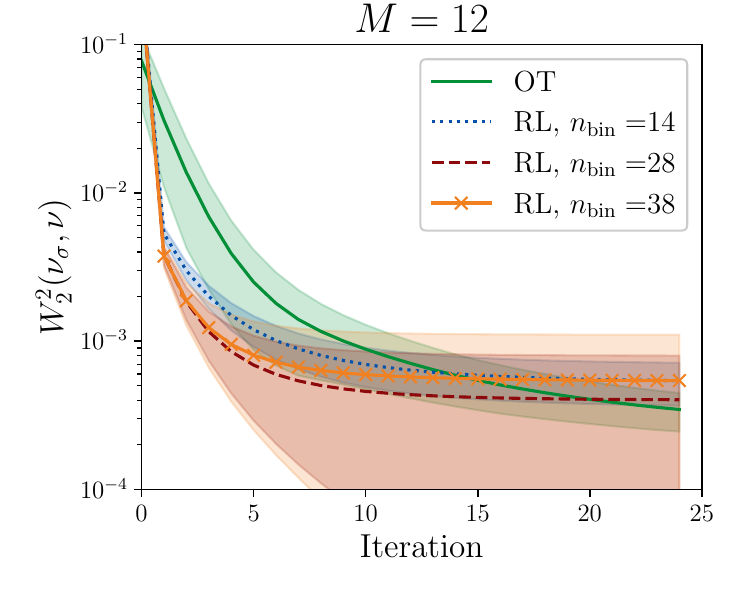}
 \caption{Behavior of $W_2(\nu_\sigma, \nu)$ along iterations for OT and RL unfolding methods, averaged across 40 independent discretizations of the same one dimensional continuum  unfolding problem. While OT requires more iterations to converge, it achieves higher accuracy. As $M$ increases, the accuracy of RL approaches that of OT, and all methods require more iterations to converge.}
\label{1dDifferentMs}
\end{center}
\end{figure}
  
  Figure \ref{1dDifferentMs} considers the decay of $W_2(\nu_\sigma, \nu)$ along iterations of OT and RL unfolding for finer discretizations of the noise model $M = 2$, 4, and 12. As in  Figure \ref{M1manyseeds}, the value of $W_2(\nu_\sigma,\nu)$ is averaged over 40 independent discretizations of the continuum  problem, which mean behavior show in solid lines and standard deviation shown in the shaded regions. As in Figure \ref{M1manyseeds},  OT requires more iterations to converge, but ultimately achieves the best accuracy. However, the benefit of OT in terms of accuracy is largest when  $M$  is small, so that the discretization of the Markov kernel is coarse. As $M$ increases, the accuracy of RL approaches that of OT, while the accuracy of OT remains roughly the same. As $M$ increases, all methods require more iterations to converge. 
  
 In practice, we also observed similar phenomena when $M$ is fixed and we considered different values of $M'$: OT outperforms for small values and is more similar to RL for large values. We believe these trends are due to the fact that smaller values of $M$ and $M'$ lead to sparser discretizations of $\nu$ and $\nu_\sigma$, which have less meaningful overlap, thus requiring coarser binning for RL to compare their relative values. Such coarse binning naturally introduces discretization errors. On the other hand, the OT method does not require absolute continuity/overlap of the measures, and its performance appears more stable to different choices of $M$ and $M'$.

\begin{figure}[htbp]
\begin{center}
\includegraphics[scale=0.40,trim={.4cm .5cm .6cm .4cm},clip]{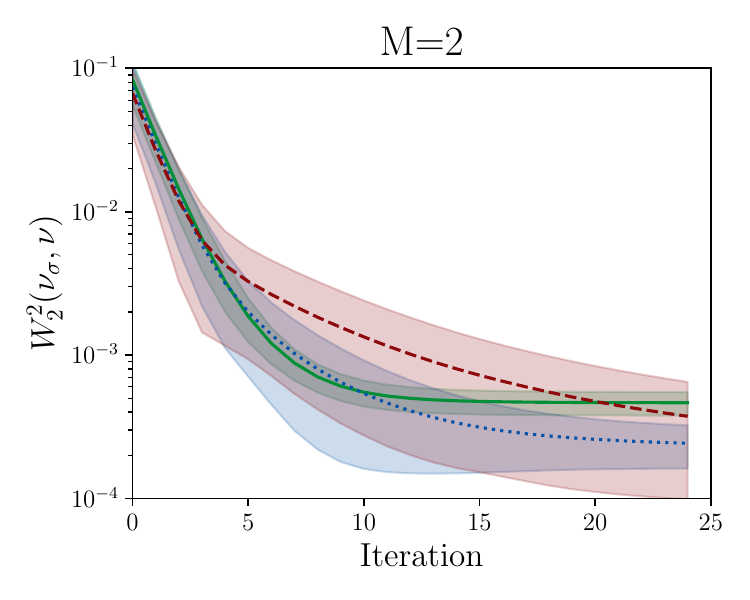}
\includegraphics[scale=0.40,trim={2.2cm .5cm .6cm .4cm},clip]{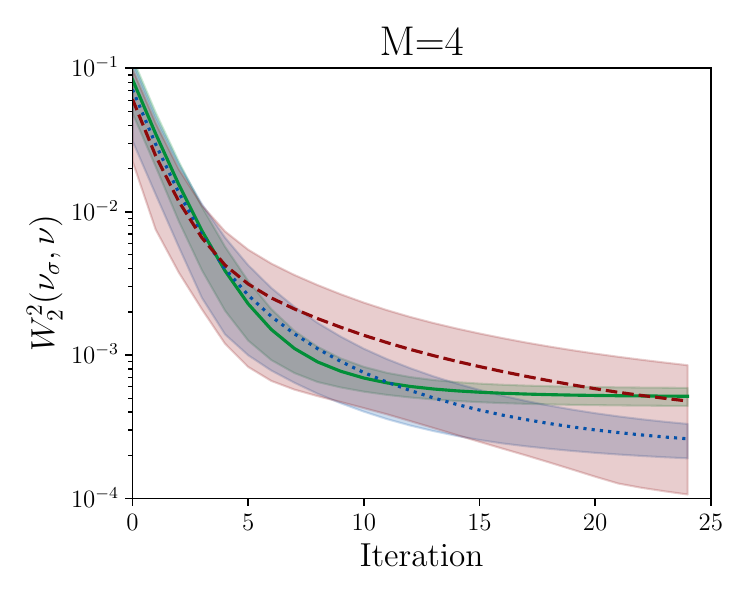}
\includegraphics[scale=0.40,trim={2.2cm .5cm .6cm .4cm},clip]{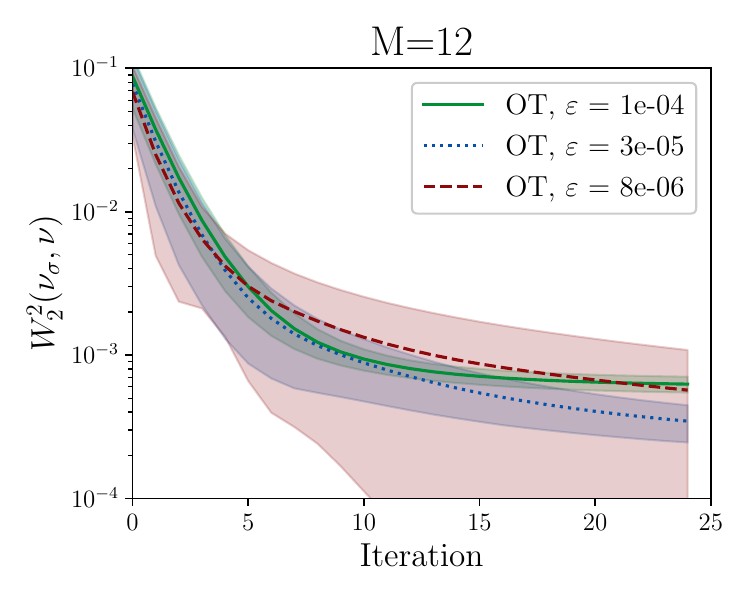}
 \caption{Behavior of $W_2(\nu_\sigma, \nu)$ along iterations of the OT  unfolding method for different choices of the Sinkhorn regularization $\epsilon$, averaged across 40 independent discretizations of the same one dimensional continuum  unfolding problem. We observe best accuracy for $\epsilon = 3 \times 10^{-5}$.}
\label{1dDifferenteps}
\end{center}
\end{figure}

  We conclude by illustrating the effect of different choices of the Sinkhorn regularization parameter $\epsilon$. In Figure \ref{1dDifferenteps} we compare the behavior of $W_2(\nu_\sigma, \nu)$ along iterations for the OT unfolding method for three choices of $\epsilon$, averaging across 40 independent discretizations of the same continuum unfolding problem. In all three cases, we observe that $\epsilon = 3 \times 10^{-5}$ achieves the best accuracy, motivating our choice of this value of $\epsilon$ throughout this manuscript.

  \subsection{Two dimensions} We now consider the behavior of OT and RL   on the two dimensional jet mass unfolding problem described in section \ref{2dsetup}, with $L=L' = 100$ and $M' = 3$.

\begin{figure}[htbp]
\begin{center}
\includegraphics[scale=0.51,trim={.4cm .5cm .55cm .4cm},clip]{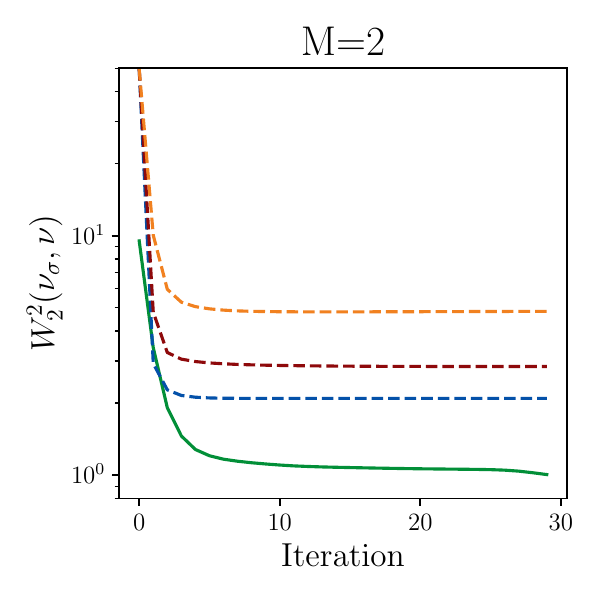}
\includegraphics[scale=0.51,trim={1.8cm .5cm .55cm .4cm},clip]{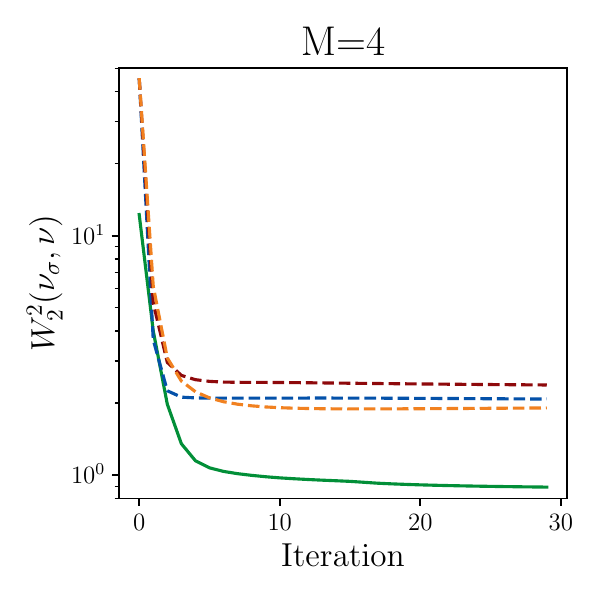}
\includegraphics[scale=0.51,trim={1.8cm .5cm .55cm .4cm},clip]{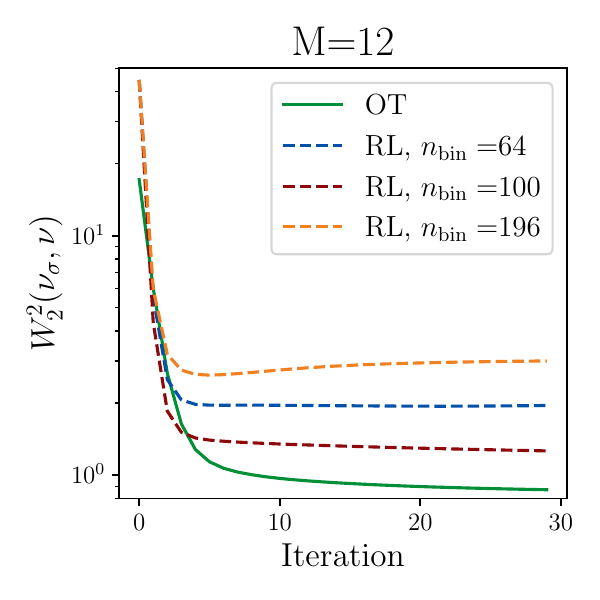}
 \caption{Behavior of $W_2(\nu_\sigma, \nu)$ along iterations for OT and RL unfolding methods on a two dimensional jet mass unfolding problem. Similar trends emerge as in the one dimensional case: OT requires more iterations to converge, but achieves higher accuracy, particularly when $M$ is small.}
\label{2dDifferentMs}
\end{center}
\end{figure}

 Figure \ref{2dDifferentMs} compares the behavior of OT and RL in terms of the value of $W_2^2(\nu_\sigma,\nu)$ along iterations, for varying discretizations of the noise model: $M=2, 4$, and 12. As in the one dimensional case, 
 we observe that RL  converges more quickly for all choices of $n_\text{bin}$, while OT achieves comparable or better accuracy in terms of $W_2^2(\nu_\sigma,\nu)$. The benefit of OT in terms of accuracy is largest when $M$ is small, so that the discretization of the Markov kernel is coarse. As $M$ increases, the accuracy of RL approaches that of OT, while the accuracy of OT remains roughly the same.  The performance of RL is non-monotonic in the number of bins, with $n_\text{bin} =100$ achieving the  best  overall accuracy across all three simulations.

\begin{figure}[htbp]
\includegraphics[valign=t,scale=0.61,trim={.5cm 5.1cm .45cm .4cm},clip]{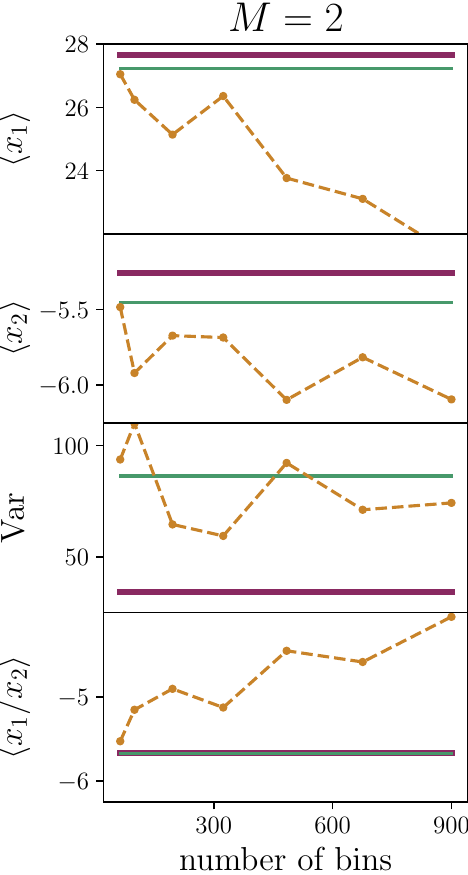}
\includegraphics[valign=t,scale=0.61,trim={2.25cm 5.1cm .45cm .4cm},clip]{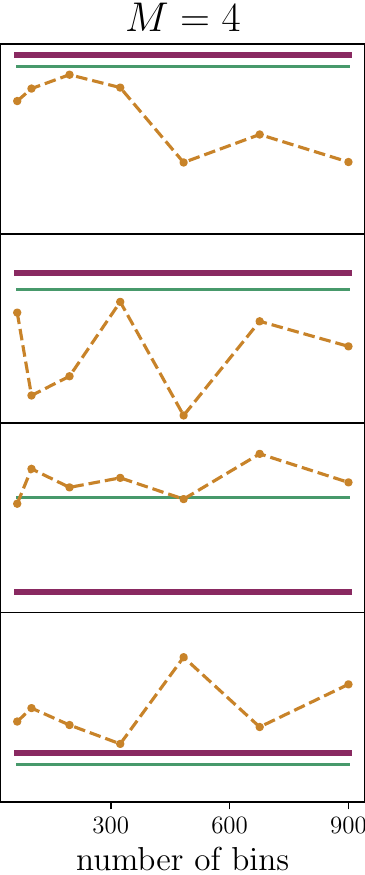}
\includegraphics[valign=t,scale=0.61,trim={2.25cm 5.1cm .45cm .4cm},clip]{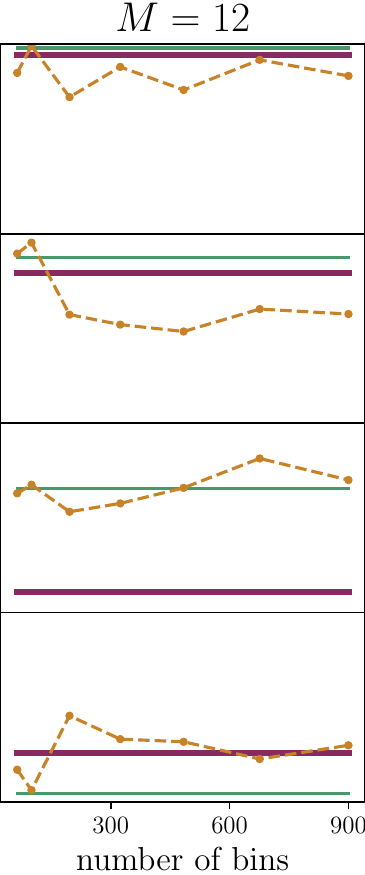} \\
\includegraphics[valign=t,scale=0.15,trim={-6.2cm 0cm 0cm -1cm},clip]{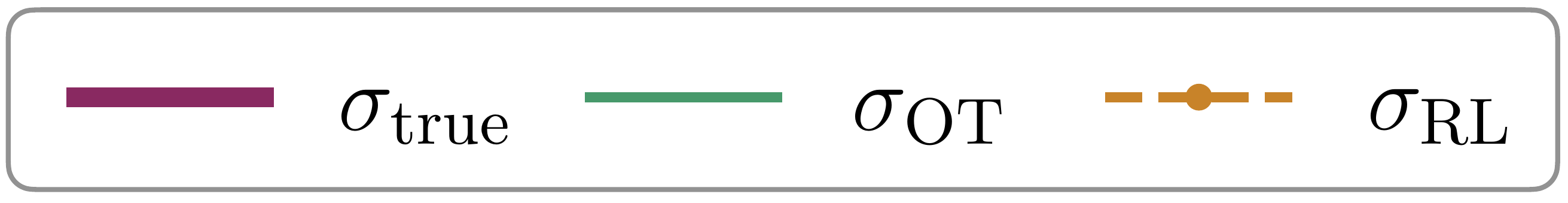} 
 \caption{Behavior of different observables along iterations for OT and RL unfolding methods on a two dimensional jet mass unfolding  problem. As in the one dimensional case, the accuracy of RL is strongly dependent on the choice of bins. In general, the OT method achieves higher overall accuracy, particularly when $M$ is small.}
\label{2dDifferentMs_observables}
\end{figure}

Figure \ref{2dDifferentMs_observables} compares the behavior of OT and RL in terms of four different observables, where for any probability measure $\mu$ and function $f(x_1, x_2)$  we denote $\langle f(x_1, x_2)  \rangle_{\mu} := \int f(x_1, x_2) d \mu(x_1,x_2)$ and    let ${\rm Var} := \langle \|(x_1,x_2)\|^2 \rangle_\mu - \| \langle (x_1,x_2) \rangle_\mu \|^2$. In the motivating jet mass problem, the moments of the true distribution $\sigma_{\text{true}}$ are more precisely known than the full distribution, and the ratio $\la x_1/x_2 \ra_{\sigma_{\text{true}}}$  is particularly sensitive to the details of hadronization --- the process by which quarks and gluons form bound states like the proton and neutron.

The solid red line shows the true value of observables for the measure ${\sigma_{\text{true}}}$.  The  solid green line shows the values for ${\sigma_{\text{OT}}}$, which are independent of choice of binning. The  dotted yellow line shows the values for ${\sigma_{\text{RL}}}$, which depends strongly on the number of bins. In agreement with the earlier simulations, which measured error in terms of $W_2(\nu_\sigma,\nu)$, we observe that RL remains sensitive to the choice of binning, and OT achieves higher overall accuracy. As before, the accuracy benefits of OT are highest for smaller values of $M$. 

\section*{Acknowledgements} The work of K. Craig and B. Faktor has  been supported in part by NSF DMS grant 2145900 and a UCSB STEM Seed Grant.  B. Nachman is supported by the DOE under contract DE-AC02-76SF00515. The authors gratefully acknowledge the support from the Simons Center for Theory of Computing, at which part of this work was completed.

The authors would like to thank L\'ena\"ic Chizat, Nathaniel Craig, Nicol\'as Garc\'ia Trillos, and Yunan Yang for helpful discussions that improved this work.
  
\bibliographystyle{siamplain}
\bibliography{references}

\appendix

\section{Supplementary Material} 
\subsection{From Bregman Projections to Sinkhorn Iterations} 
We begin by showing the equivalence of the Bregman projections \eqref{Bregman iterations} and Sinkhorn iterations \eqref{Modified Sinkhorn iterations which are equivalent to Bregman}. Namely, we will prove the following theorem.

\begin{theorem}
\label{Bregman-Sinkhorn equivalence}
Fix ${\bu}^0 \in \mathbb{R}^{m+L}_+$, ${\bv}^0 \in \mathbb{R}^{n+1}_+$ and set ${\bP}^0 = {\rm diag}({\bu}^0) \tilde{\bK} {\rm diag}({\bv}^0)$. Then, for all $l \geq 0$,  the iterations (\ref{Bregman iterations})-(\ref{Modified Sinkhorn iterations which are equivalent to Bregman}) satisfy 
\begin{align*}
\bP^{2l+1} =  {\rm diag}(\bu^{l+1}) \tilde{\bK} {\rm diag}(\bv^l), \quad \bP^{2l+2} =  {\rm diag}(\bu^{l+1}) \tilde{\bK}  {\rm diag}(\bv^{l+1}).
\end{align*}
\end{theorem}
This theorem has a simple interpretation: updating $\bu$   to ensure that  $ {\rm diag}(\bu) \tilde{\bK} {\rm diag}(\bv)$ has the correct first marginal is equivalent to updating   ${\rm diag}(\bu) \tilde{\bK} {\rm diag}(\bv)$  to ensure that it itself has the correct first marginal (when both updates are done in the least KL-costly manner). 
A parallel interpretation holds regarding second marginals. 

Before we prove Theorem \ref{Bregman-Sinkhorn equivalence}, we begin with the following proposition on the form of \textbf{KL}-projections onto the subspace having first marginal in ${\rm \ker} \, \bB$.

\begin{proposition}
\label{Prop Form of projection onto matrices with first marginal in ker B part 2}
For every $\bP \in \mathbb{R}_+^{(m+L) \times (n+1)}$, we have
\begin{align} \label{projectionproposition}
{\rm Proj}_{\mathcal{A}}^{\mathbf{KL}}(\bP) = {\rm diag} \left(  \frac{ { \rm Proj}_{{\ker \bB}}^{\mathbf{KL}}(\bP \mathbf{1}_{n+1})}{\bP \mathbf{1}_{n+1}} \right) \bP .
\end{align}
\end{proposition}
\begin{proof}
 
We associate to the  minimzation
$${\rm Proj}_{\mathcal{A} }^{\mathbf{KL}}(\bP)  =  \argmin_{\substack{\bP' \\ \bP' \mathbf{1}_{n+1} \in \ker(\bB) }}  \textbf{KL}(\bP' | \bP)$$
the Lagrangian variable $\mathbf{f}$ and Lagrangian
$$
\mathcal{L}(\bP', \bf) = \textbf{KL}(\bP'|\bP) - \bf^T\bB\bP'\mathbf{1}_{n+1}.
$$
Analysis of the stationary points gives, at optimum $\bP'_*$, that $\bP'_* = {\rm diag} (\bg_*) \bP$ where $\bg_* = e^{\bB^T \bf_*}$. It remains to show that 
$\bg_*=    { \rm Proj}_{{\ker \bB}}^{\mathbf{KL}}(\bP \mathbf{1}_{n+1}) / \bP \mathbf{1}_{n+1}$, or equivalently, that $\bw_*:= \bg_* \odot {\bP}  \mathbf{1}_{n+1} = {\rm Proj}_{{\ker B}}^{\mathbf{KL}}(\bP \mathbf{1}_{n+1})$.
 
Restricting the minimization over the $\bP'$ to be over those of the form $\bP' = {\rm diag}(\bg)\bP$, we find
\begin{align*}
\bg_* =& \argmin_{\substack{ \bg  \\
 {\rm diag}(\bg) \bP \mathbf{1}_{n+1} \in \ker \bB }}   \sum_{r,s} \bg_r \bP_{rs} \log(\bg_r) - \bg_r \bP_{rs} .\end{align*}
Restricting to those $\bg$ of the form $\bg = \bw/\bP \mathbf{1}_{n+1}$ we obtain
 \begin{align*}
\bw_* =& \argmin_{\substack{\bw  \\ \bw \in \ker \bB}}    \sum_r \bw_r \log \left( \frac{\bw_r}{\sum_s \bP_{rs}}\right) -   \bw_r   =  {\rm Proj}_{\ker \bB}^{\mathbf{KL}}(\bP \mathbf{1}_{n+1})
\end{align*}
 \end{proof}

Now, we turn to the proof of Theorem \ref{Bregman-Sinkhorn equivalence}.

\begin{proof}[Proof of Theorem \ref{Bregman-Sinkhorn equivalence}]
Assume $\bP^{2l} =  {\rm diag}(\bu^l) \tilde{\bK} {\rm diag}(\bv^l)$ for some $l \geq 0$, which holds by assumption for $l=0$. 
Then $\bP^{2l} \mathbf{1}_{n+1} = \bu^l \odot \tilde{\bK} \bv^l$, so Proposition \ref{Prop Form of projection onto matrices with first marginal in ker B part 2} gives
\begin{align*}
\bP^{2l+1} = {\rm diag}\left( \frac{ {\rm Proj}_{{\rm ker} \bB}^{\mathbf{KL}}(\bu^l \odot \tilde{\bK} \bv^l)}{\bu^l \odot \tilde{\bK} \bv^l}\right) {\rm diag}(\bu^l) \tilde{\bK} {\rm diag}(\bv^l) = {\rm diag}(\bu^{l+1}) \tilde{\bK} {\rm diag}(\bv^l).
\end{align*}
By using the explicit expression for the \textbf{KL} projection onto the subspace with prescribed second marginal \cite[Proposition 1]{Be15}, we likewise have 
\begin{align*}
\bP^{2l+2} =& \bP^{2l+1} {\rm diag}\left( \frac{\bb}{(\bP^{2l+1})^T \mathbf{1}_{m+L}}\right) = {\rm diag}(\bu^{l+1}) \tilde{\bK} {\rm diag}(\bv^l) {\rm diag} \left( \frac{\bb}{\bv^l \odot \tilde{\bK}^T \bu^{l+1}} \right) \\
=& {\rm diag}(\bu^{l+1}) \tilde{\bK} {\rm diag}(\bv^{l+1}).
\end{align*}
\end{proof}

Next, we  justify our choice of initialization for our iterative method by showing that, if our prior information coincides with the optimum value, the iterations remain at the optimizer.

\begin{proposition}
Let $\bP_*$ denote the minimizer of equation \eqref{minimization problem discrete entropy regularized in Bregman/Sinkhorn form}, and let  $\bsigma^0 \in \mathbb{R}^L_+$ be the associated true data, $ \bsigma^0_k := (\mathbf{P}_{*})_{m+k,n+1}$ for $ k = 1, \dots ,L $. 
Let $\bgamma_*$ denote the optimizer of the classical entropy regularized optimal transport problem  between $\bR \bsigma^0$ and $\bnu$, as in equation \eqref{OTproblemforcomputingprior}, which is of the form $\boldsymbol{\gamma}_0^* = {\rm diag}(\bc) \bK {\rm diag}(\bd)$ for some $\bc \in \mathbb{R}^m_+$ and $\bd \in \mathbb{R}^n_+$.

Define $\bu^0 := [\bc ; \bsigma^0]$. Then computing $\bu^l$ and $\bv^l$ according to the generalized Sinkhorn iterations (\ref{Modified Sinkhorn iterations which are equivalent to Bregman}), we have  $\mathbf{P}^{2l} ={\rm diag}(\mathbf{u}^{l}) \tilde{\mathbf{K}}  {\rm diag}(\mathbf{v}^{l})$ satisfies $\bP^l \equiv \bP^0 = \bP_*$ for all $l \geq 0$.
\end{proposition}
\begin{proof}
Observe
$$
\bP^0 := {\rm diag}(\bu^0) \tilde{\bK} {\rm diag}(\bv^0)
= \begin{bmatrix} {\rm diag}(\bc) & \\ & {\rm diag}(\bsigma^0) \end{bmatrix}
\begin{bmatrix} \bK & \\ & \mathbf{1}_L \end{bmatrix}
{\rm diag} \left( \frac{\begin{bmatrix} \bnu \\ 1 \end{bmatrix}}{\begin{bmatrix} \bK^T & \\ & \mathbf{1}_L^T \end{bmatrix} \begin{bmatrix}\bc \\ \bsigma^0 \end{bmatrix} } \right)
$$
may be simplified as
$$
\begin{bmatrix} {\rm diag}(\bc) \bK {\rm diag}\left(\frac{\bnu}{\bK^T \bc}\right) & \\ & \bsigma^0\end{bmatrix},
$$
and ${\rm diag}(\bc) \bK {\rm diag}(\bnu / \bK^T \bc)$ is precisely $\bgamma_*$ by the required marginal constraint.
Hence 
$\bP^0 \in \mathcal{A} \cap \mathcal{B}$. This shows that $\bP^l \equiv \bP^0 $ for all $l \geq 0$. Finally, the fact that the generalized Sinkhorn iterations converge to optimum (see Remark \ref{rate of convergence}) gives the result. 
\end{proof}

We conclude by proving Proposition \ref{Peyre-Cuturi 4.3 analog}, which shows that the minimizer of the discrete, entropically regularized problem (\ref{new-entropically-regularized-discrete-denoising-problem}) can be expressed in terms of lower dimensional scaling variables.
\begin{proof}[Proof of Proposition \ref{Peyre-Cuturi 4.3 analog}]
The Lagrangian is 
\begin{align*}
\mathcal{L}((\mathbf\Gamma, \bsigma), \mathbf{f}, \mathbf{g}) =& \langle \mathbf\Gamma, \mathbf{C} \rangle 
- \varepsilon \mathbf{H}(\bsigma) - \varepsilon \mathbf{H}(\mathbf\Gamma) \\
&- \left\langle \mathbf{f}, \left[\sum_{j=1}^n \mathbf\Gamma_{ij} - \sum_{k=1}^L \mathbf{R}_{ik} \bsigma_k\right]_i \right\rangle - \left\langle \mathbf{g}, \left[ \sum_{i=1}^m \mathbf\Gamma_{ij} -\bnu_j\right]_j \right\rangle
\end{align*}
where $\mathbf{f} \in \mathbb{R}^m$, $\mathbf{g} \in \mathbb{R}^n$ are the Lagrange multipliers. At stationary points 
$$
0 = \frac{\partial \mathcal{L}}{\partial \mathbf\Gamma_{ij}} = \mathbf{C}_{ij} + \varepsilon \log(\mathbf\Gamma_{ij}) - \mathbf{f}_i - \mathbf{g}_j \text{ and }
0 = \frac{\partial \mathcal{L}}{\partial \bsigma_k} = 
\varepsilon \log(\bsigma_k) + \sum_{i=1}^m \mathbf{R}_{ik} \mathbf{f}_i.
$$
Thus,
$
\mathbf\Gamma_{ij} = e^{\mathbf{f}_i / \varepsilon} e^{-\mathbf{C}_{ij} / \varepsilon} e^{\mathbf{g}_j / \varepsilon}.
$
Setting $\mathbf{u}_i = e^{\mathbf{f}_i/\varepsilon}$, and $\mathbf{v}_j = e^{\mathbf{g}_j/\varepsilon}$ completes the proof.
\end{proof}

\end{document}